\documentclass[12pt, a4paper]{amsart}
\usepackage{amsfonts,latexsym,graphicx,
amssymb,amsthm,amsmath}

\newtheorem{theorem}{Theorem}[section]
\newtheorem{proposition}[theorem]{Proposition}
\newtheorem{lemma}[theorem]{Lemma}
\newtheorem{corollary}[theorem]{Corollary}
\newtheorem{remark}[theorem]{Remark}
\newtheorem{main-theorem}[theorem]{Main Theorem}

\title[The 
Iwasawa, Bruhat decompositions
of $\mathrm{F}_{4(-20)}$]
{The Iwasawa decomposition
and the  Bruhat decomposition of
the automorphism group
on certain exceptional Jordan algebra.}
\author{Akihiro Nishio}
\date{2013/02/11}
\begin{document}

\begin{abstract}
Let $\mathcal{J}^1$ be the real form of 
a complex simple Jordan algebra
such that the automorphism group is $\mathrm{F}_{4(-20)}$.
By using some orbit types
of $\mathrm{F}_{4(-20)}$ on $\mathcal{J}^1$,
for $\mathrm{F}_{4(-20)}$,
explicitly, we give
the Iwasawa decomposition, the 
Oshima--Sekiguchi's
$K_{\epsilon}-$Iwasawa
decomposition,
the Matsuki decomposition, 
and the  Bruhat and Gauss decompositions.
\end{abstract}


\address{
Graduate~school~of~Engineering,
University~of~Fukui,
Fukui-shi, 910-8507, Japan}

\email{nishio@quantum.apphy.u-fukui.ac.jp}

\maketitle

This article is a continuation of \cite{Na2012}.
\begin{center}
{\large Contents}
\end{center}

\quad\ \ \ref{itd}. Overview.
\smallskip

\quad\  \ref{prl2}. Preliminaries.
\smallskip

\quad\  \ref{id}. 
The Iwasawa decomposition of 
$\mathrm{F}_{4(-20)}$.
\smallskip

\quad\  \ref{ke}. 
The $K_{\epsilon}-$Iwasawa decomposition
of $\mathrm{F}_{4(-20)}$
\smallskip

\quad\ \ref{md}. 
The Matsuki decomposition
of $\mathrm{F}_{4(-20)}$
\smallskip

\quad\  \ref{gd}. The Bruhat and Gauss 
decompositions of $\mathrm{F}_{4(-20)}$.
\smallskip

\quad\  Appendix A. 
The 
explicit formula $c$-function
of $\mathrm{F}_{4(-20)}$

\addtocounter{section}{+8}
\section{Overview.}\label{itd}
Let $G$ be a connected non-compact
semisimple $\mathbb{R}-$Lie group
of which
the center $Z(G)$ is finite.
We denote its $\mathbb{R}-$Lie algebra 
by $\mathfrak{g}=Lie(G)$.
Let $\theta$ be a Cartan involution of $\mathfrak{g}$
and its Cartan decomposition
$\mathfrak{g}=\mathfrak{k}\oplus \mathfrak{p}$
where
$\mathfrak{k}:
=\{X\in\mathfrak{g}|~\theta X=X\}$ and
$\mathfrak{p}:
=\{X\in\mathfrak{g}|~\theta X=-X\}$.
Let
$\mathfrak{a}$ be a maximal abelian subspace 
of $\mathfrak{p}$,
$\mathfrak{a}^*$
the dual space of $\mathfrak{a}$,
and  
$\mathfrak{m}=Z_{\mathfrak{k}}(\mathfrak{a})$
the centralizer of the subset $\mathfrak{a}$ of 
the Lie algebra $\mathfrak{k}$.
For each  $\lambda \in \mathfrak{a}^*$, 
let
$\mathfrak{g}_{\lambda}
:=\{X\in\mathfrak{g}|~[H,X]=\lambda(H)X~
{\rm for~all}~H\in\mathfrak{a}\}$.
$\lambda$ is called a {\it root} 
of $(\mathfrak{g},\mathfrak{a})$
if $\lambda \ne 0$ 
and $\mathfrak{g}_{\lambda} \ne \{0\}$.
We denote the set of roots of $(\mathfrak{g},\mathfrak{a})$
by $\Sigma$.
Then $\mathfrak{g}
= \mathfrak{g}_0 \oplus \sum_{\lambda \in \Sigma}
\mathfrak{g}_{\lambda}$, $[\mathfrak{g}_{\lambda},
\mathfrak{g}_{\mu}] \subset \mathfrak{g}_{\lambda+\mu}$,
$\theta \mathfrak{g}_{\lambda}=\mathfrak{g}_{-\lambda}$,
and $\mathfrak{g}_0 = \mathfrak{a} \oplus \mathfrak{m}$
(cf. \cite[Ch~V]{Knp1986}).
We introduce an ordering in $\mathfrak{a}^*$,
and this ordering single out the set
$\Sigma^+$
of positive roots.
We denote
$\Sigma^-:=\{-\lambda|~\lambda\in\Sigma^+\}$,
$\mathfrak{n}^+
:=\sum_{\lambda\in\Sigma^+}\mathfrak{g}_{\lambda}$,
and $\mathfrak{n}^-
:=\sum_{\lambda\in\Sigma^-}\mathfrak{g}_{\lambda}$.
Then $\mathfrak{n}^+$ and $\mathfrak{n}^-$
are nilpotent subalgebras
such that
$\theta\mathfrak{n}^\pm=\mathfrak{n}^\mp~(resp)$ and 
$\mathfrak{g}=\mathfrak{n}^-
\oplus \mathfrak{a} \oplus \mathfrak{m} \oplus 
\mathfrak{n}^+$.
For each involutive automorphism $\varphi$ 
on $G$,
we denote the subgroup $G^{\varphi}
=\{g \in G|~\varphi (g)=g\}$ of $G$.
Let $\Theta$ be an involutive automorphism  on $G$,
of which
the differential at the identity element is the Cartan involution $\theta$
of $\mathfrak{g}$:
$d\Theta=\theta$,
and $K:=G^{\Theta}$.
Note that
$Lie(K)=\mathfrak{k}$, 
$K$ is connected and closed,
and that
$K$ is a maximal compact subgroup of $G$
(cf. \cite[Ch~VI,Theorem~1.1]{Hl2001}).
We denote the subgroups $A:=\exp \mathfrak{a}$, 
 $N^\pm:=\exp \mathfrak{n}^{\pm}~(resp)$,
 and
$M:=Z_K(\mathfrak{a})$
the centralizer of the set $\mathfrak{a}$ of  $K$,
respectively.
Then the identity connected component $M^0$ of $M$ 
is a connected Lie subgroup corresponding
to $\mathfrak{m}$, and $\Theta (N^\pm)=N^\mp~(resp)$.
We denote the normalizer of the subset  $\mathfrak{a}$
of $K$ by
$M^*:=N_K(\mathfrak{a})$,
and the finite factor group $W:=M^*/M$.
For all $w\in W$, we
fix a representative $\tilde{w}\in M^*$.
Then 
\begin{align*}
{\rm (1)} &~~ G=KAN^+ &&{\rm (Iwasawa~decomposition),}\\
{\rm (2)} &~~G=\coprod_{w\in W}
 N^-\tilde{w}MAN^+&& {\rm (Bruhat~decomposition),}\\
{\rm (2)'} &~~G
=\overline{N^-MAN^+}&& {\rm (Gauss~decomposition).}
\end{align*}
(cf. \cite{Hl2001}, \cite{Mlch1995} ).
For any $g\in G$, 
there exist
unique elements $k(g)\in K,
H(g)\in \mathfrak{a}$, and $n_I(g)\in N^+$
such that \[g=k(g)(\exp H(g)) n_I(g).\]
In ${\rm (2)'}$,  the submanifold $N^-MAN^+$ 
is open dense in $G$,
and for any $g\in N^-MAN^+$,
there exist
unique elements $n^-_G(g)\in N^-$,
$m_G(g)\in M$, $a_G(g)\in A$, and $n^+_G(g)\in N^+$
such that
\[g=n^-_G(g)m_G(g)a_G(g)n^+_G(g).\]
However, in this article, the existence and uniqueness 
of factors of Iwasawa and Gauss decompositions
for the Lie group $\mathrm{F}_{4(-20)}$
will be shown  
by using concrete  $\mathrm{F}_{4(-20)}$-orbits
and stabilizers of $\mathrm{F}_{4(-20)}$ in \cite{Na2012}.

According to \cite[Definition 1.1]{OS1980},
a {\it signature of roots} 
is defined by the mapping $\epsilon$
of $\Sigma$ to $\{-1,1\}$ such that $\epsilon$ satisfies
the following conditions:
\begin{align*}
{\rm (i)}&~~ \epsilon(\lambda)=\epsilon(-\lambda) &&
{\rm for~any}~\lambda\in \Sigma,\\
{\rm (ii)}&~~ \epsilon(\lambda+\mu)
=\epsilon(\lambda)\epsilon(\mu) &&
{\rm if}~\lambda,\mu~\text{and}~\lambda+\mu\in \Sigma.
\end{align*}
According to \cite[Definition 1.2]{OS1980},
for 
any signature $\epsilon$ of roots
with respect to  the Cartan involution $\theta$,
an involutive automorphism $\theta_{\epsilon}$ of
$\mathfrak{g}$ is defined as
\begin{align*}
{\rm (i)}&~~\theta_{\epsilon}(X)
:=\epsilon(\lambda)\theta(X)&&
{\rm for~any}~\lambda\in \Sigma
~~{\rm and}~~X\in \mathfrak{g}_{\lambda},\\
{\rm (ii)}&~~
\theta_{\epsilon}(X):=\theta(X)
&&{\rm for~any}~X\in \mathfrak{a}\oplus\mathfrak{m}.
\end{align*}
Setting $\mathfrak{k}_{\epsilon}:
=\{X\in\mathfrak{g}|~\theta_{\epsilon}X=X\}$ and
$\mathfrak{p}_{\epsilon}:
=\{X\in\mathfrak{g}|~\theta_{\epsilon}X=-X\}$,
$\mathfrak{g}=\mathfrak{k}_{\epsilon}\oplus
\mathfrak{p}_{\epsilon}$.
We denote the connected Lie subgroup
having the Lie algebra $\mathfrak{k}_{\epsilon}$
by $(K_{\epsilon})^0$.
We define the subgroup $K_{\epsilon}$ by
\[K_{\epsilon}:=(K_{\epsilon})^0M.\]
In fact, 
since
all elements of $M$ normalize $(K_{\epsilon})^0$
from
\cite[Lemma~1.4(i)]{OS1980},
$K_{\epsilon}$
is a subgroup of $G$.
We denote 
\[M_{\epsilon}^*:=K_{\epsilon}\cap M^*,\quad
W_{\epsilon}:=M_{\epsilon}^*/M.\]

\begin{proposition}\label{itd-01}
{\rm (T. Oshima and J. Sekiguchi
\cite[Proposition~1.10]{OS1980})}.
Let the factor set $W_{\epsilon}\backslash W
=\{w_1=1,w_2,\cdots ,w_r\}$ 
where $r=[W:W_{\epsilon}]$.
Fix representatives
$\tilde{w}_1=1,\tilde{w}_2,\cdots ,\tilde{w}_r
\in M_{\epsilon}^*
=K_{\epsilon}\cap M^*$
for $w_1=1,w_2,\cdots ,w_r$.
Then the decomposition
\[G\supset \cup_{i=1}^rK_{\epsilon}\tilde{w}_iAN^+\]
has the following properties.

\noindent
{\rm (1)} If $k\tilde{w}_ian=k'\tilde{w}_ja'n'$ 
with $k,k'\in K_{\epsilon},$
$a,a'\in A$, and $n,n'\in N^+$, 
then $k=k',$ $i=j,$ $a=a'$, and $n=n'$.
\smallskip

\noindent
{\rm (2)} The map $(k,a,n)\mapsto k\tilde{w}_ian$ 
defines an analytic
diffeomorphism of the product manifold 
$K_{\epsilon}\times
A\times N^+$ onto the open submanifold 
$K_{\epsilon}\tilde{w}_iAN^+$ of
$G~(i=1,\cdots r)$. 
\smallskip

\noindent
{\rm (3)} The submanifold
$\cup_{i=1}^rK_{\epsilon}\tilde{w}_iAN^+$
is open dense in $G$.
\end{proposition}
The decomposition
$G=\overline{\cup_{i=1}^rK_{\epsilon}\tilde{w}_iAN^+}$
is called the $K_{\epsilon}-${\it Iwasawa decomposition of}
$G$.
\medskip

If a group $G$ acts on a set $S$,
we denote the pointwise stabilizer of finite set 
$\{x_1,\cdots,x_n\}$ of $S$
by $G_{x_1,\cdots,x_n}
:=\{g \in G|~g x_i=x_i~\text{for}~i=1,\cdots,n\}$,
and the $G$-orbit of $x\in S$ by
$G \cdot x:=\{g x|~g\in G\}$.
We denote the Kronecker delta by $\delta_{i,j}$.
Let ${\bf O}$ be the octonions
having the conjugation $\overline{x}$ and
inner product $(x|y)$ for $x,y \in {\bf O}$.
We denote the natural unit octonions:
$\{1(=e_0),e_1,e_2,e_3,e_4,e_5,e_6,e_7\}$.
Set
\[
h^1(\xi_1,\xi_2,\xi_3;x_1,x_2,x_3):=
\begin{pmatrix}
\xi_1 & \sqrt{-}x_3 & \sqrt{-1}\overline{x_2}\\
\sqrt{-1}\overline{x_3} & \xi_2 & x_1\\
\sqrt{-1}x_2 & \overline{x_1} & \xi_3
\end{pmatrix}
\]
with $\xi_i \in \mathbb{R}$, $x_i \in {\bf O}$.
In \cite[\S 1]{Na2012},
the {\it exceptional Jordan algebra} $\mathcal{J}^1$
is given by 
\[\mathcal{J}^1:=\{h^1(\xi_1,\xi_2,\xi_3;x_1,x_2,x_3)|
~\xi_i \in \mathbb{R},~x_i \in {\bf O}\}\]
with the {\it Jordan product} $X\circ Y=2^{-1}(XY+YX)$ 
for $X,Y \in\mathcal{J}^1$.
Put
$E=h^1(1,1,1;0,0,0)$,
$E_i:=h^1(\delta_{i,1},\delta_{i,2},\delta_{i,3};0,0,0)$,
and
$F_i^1(x):=h^1(0,0,0;
\delta_{i,1}x,\delta_{i,2}x,\delta_{i,3}x)$.
Then
$h^1(\xi_1,\xi_2,\xi_3;x_1,x_2,x_3)
=\sum{}_{i=1}^3(\xi_i E_i+F_i^1(x_i))$.
We recall that
$\mathcal{J}^1$ has the {\it trace}
$\mathrm{tr}(X):=\sum_{i=1}^3\xi_i$
where $X=\sum_{i=1}^3
(\xi_i E_i+F_i^1(x_i))$,
the {\it inner product} $(X|Y):=\mathrm{tr}(X\circ Y)$,
the {\it cross product} $X\times Y$ by
\[X\times Y:=2^{-1}\bigl(2X\circ Y
-\mathrm{tr}(X)Y-\mathrm{tr}(Y)X
+(\mathrm{tr}(X)\mathrm{tr}(Y)-(X|Y))E\bigr)\] 
as well as $X^{\times 2}:=X\times X$,
and the {\it determinant} $\mathrm{det}(X):=3^{-1}
(X|X^{\times 2})$, respectively.
By \cite[Lemma~1.6]{Na2012},
\begin{gather*}
\tag{\ref{itd}.1}\label{itd001}
(X|Y)=
(\sum{}_{i=1}^3\xi_i\eta_i)
+2(x_1|y_1)-2(x_2|y_2)-2(x_3|y_3),\\
\mathrm{det}(X)=\xi_1\xi_2\xi_3
-2(1|(x_1x_2)x_3)
-\xi _1(x_1|x_1)+\xi_2(x_2|x_2)+\xi_3(x_3|x_3),\\
\tag{\ref{itd}.2}\label{itd002}
X^{\times 2}=\bigl(\xi_2\xi_3-(x_1|x_1)\bigr)E_1
+\bigl(\xi_3\xi_1+(x_2|x_2)\bigr)E_2
+\bigl(\xi_1\xi_2+(x_3|x_3)\bigr)E_3\\
+F_1^1(-\overline{x_2x_3}-\xi_1 x_1)
+F_2^1(\overline{x_3x_1}-\xi_2 x_2)
+F_3^1(\overline{x_1x_2}-\xi_3 x_3)
\end{gather*}
where $X=\sum{}_{i=1}^3(\xi_i E_i+F_i^1(x_i))$
and $Y=\sum{}_{i=1}^3(\eta_i E_i+F_i^1(y_i))$.
We recall that $\mathcal{J}^1$
has the
{\it exceptional hyperbolic planes} $\mathcal{H}$,
$\mathcal{H'}$
and the {\it exceptional null cones} $\mathcal{N}_1^+$,
$\mathcal{N}_1^-$ as
\begin{align*}
\mathcal{H}&:=\{
X \in \mathcal{J}^1|~X^{\times 2}=0,~\mathrm{tr}(X)=1,
~(E_1|X)\geq 1\},\\
\mathcal{H'}&:=\{
X \in \mathcal{J}^1|~X^{\times 2}=0,~\mathrm{tr}(X)=1,
~(E_1|X)\leq 0\},\\
\mathcal{N}_1^+&:=\{
X \in \mathcal{J}^1|~X^{\times 2}=0,~\mathrm{tr}(X)=0,
~(E_1|X)> 0\},\\
\mathcal{N}_1^-&:=\{
X \in \mathcal{J}^1|~X^{\times 2}=0,~\mathrm{tr}(X)=0,
~(E_1|X)< 0\}.
\end{align*}
respectively.
In Lemma~\ref{prl2-17}, we will show the following equations:
\[
\tag{\ref{itd}.3}\label{itd003}
\left\{
\begin{array}{rl}
{\rm (i)}&
\{X \in \mathcal{H}|~(E_1|X)=1\}=\{E_1\},\\
{\rm (ii)}&
\{X \in \mathcal{H}'|~(E_1|X)=0\}
=2^{-1}( S^8+(E_2+E_3))
\supset
\{E_2,E_3\}
\end{array}
\right.
\]
where $S^8=\{\xi(E_2-E_3)+F_1^1(x)|~\xi^2+(x|x)=1\}$.

The {\it exceptional Lie group}
$\mathrm{F}_{4(-20)}$ is given by
\[
\mathrm{F}_{4(-20)}:=\{g \in \mathrm{GL}_{\mathbb{R}}
(\mathcal{J}^1)|~g(X \circ Y)=g X \circ g Y\}\]
which satisfies that
\begin{gather*}
\mathrm{tr}(g X)=\mathrm{tr}(X),
\quad g E=E,\quad
(g X|g Y)=(X|Y),\\
g(X \times Y)=g X \times g Y,
\quad \mathrm{det}(g X)=\mathrm{det}(X)
\end{gather*}
for all $g \in \mathrm{F}_{4(-20)}$
and $X,Y \in \mathcal{J}^1$,
from \cite[Proposition~1.8]{Na2012}.
In \cite[Theorem~2.2.2]{Yi1990} and
\cite[Theorem~2.14.1]{Yi_arxiv},
I. Yokota has proved that
$\mathrm{F}_{4(-20)}$ is connected and 
a simply connected semisimple Lie group 
of type ${\bf F}_{4(-20)}$,
by showing the polar decomposition
$\mathrm{F}_{4(-20)}\simeq
\mathrm{Spin}(9)\times \mathbb{R}^{16}$
with the center $Z(\mathrm{F}_{4(-20)})=\{1\}$
(\cite[Theorem~2.14.2]{Yi1990}).
We denote the elements $P^+,P^-\in \mathcal{J}^1$ by
$P^+:=h^1(1,-1,0;0,0,1)$
and 
$P^-:=h^1(-1,1,0;0,0,1)$
respectively.
From \cite[Proposition~0.1]{Na2012},
we recall that
the exceptional hyperbolic planes
and the exceptional null cones
are $\mathrm{F}_{4(-20)}$-orbits
in $\mathcal{J}^1$:
\begin{align*}
\tag{\ref{itd}.4}\label{itd004}
\mathcal{H}&=\mathrm{F}_{4(-20)} \cdot E_1,\\
\tag{\ref{itd}.5}\label{itd005}
\mathcal{H'}&=\mathrm{F}_{4(-20)} \cdot E_2
=\mathrm{F}_{4(-20)} \cdot E_3,\\
\tag{\ref{itd}.6}\label{itd006}
\mathcal{N}_1^+&=\mathrm{F}_{4(-20)} \cdot P^+,\\
\tag{\ref{itd}.7}\label{itd007}
\mathcal{N}_1^-&=\mathrm{F}_{4(-20)} \cdot P^-.
\end{align*}
For $i \in \{1,2,3\}$,
we denote the element $\sigma_i \in \mathrm{F}_{4(-20)}$ by
\[\sigma_i \left(\sum{}_{j=1}^3(\xi_j E_j+F_j^1(x_j))\right)
:=\sum{}_{j=1}^3\left(\xi_j E_j+
F_j^1\bigl((-1)^{1-\delta_{i,j}}x_j\bigr)\right).
\]
(see \cite[\S 4]{Na2012}), and 
the involutive inner automorphism 
$\tilde{\sigma}_i$;
$\tilde{\sigma}_i(g):=\sigma_i g \sigma_i^{-1}
=\sigma_i g \sigma_i$
for $g \in \mathrm{F}_{4(-20)}$.
We simply write  $\sigma$ and $\tilde{\sigma}$
for $\sigma_1$ and $\tilde{\sigma}_1$,
respectively.
Set $(G,\Theta)=(\mathrm{F}_{4(-20)},\tilde{\sigma})$
and $K:=(\mathrm{F}_{4(-20)})^{\tilde{\sigma}}$.
From \cite[Proposition~4.8]{Na2012}
(note $(\mathrm{F}_{4(-20)})_{E_2}\cong
(\mathrm{F}_{4(-20)})_{E_3}$),
the stabilizers
$(\mathrm{F}_{4(-20)})_{E_1}$
and $(\mathrm{F}_{4(-20)})_{E_2}$
are connected two-hold covering groups of
$\mathrm{SO}(9)$ and $\mathrm{SO}^0(8,1)$, respectively.
So we denote $\mathrm{Spin}(9):=(\mathrm{F}_{4(-20)})_{E_1}$
and $\mathrm{Spin}^0(8,1):=(\mathrm{F}_{4(-20)})_{E_2}$, respectively.
By \cite[Proposition~4.14]{Na2012},
\begin{gather*}
\tag{\ref{itd}.8}\label{itd008}
K=(\mathrm{F}_{4(-20)})_{E_1}=\mathrm{Spin}(9).\\
\tag{\ref{itd}.9}\label{itd009}
(\mathrm{F}_{4(-20)})^{\tilde{\sigma}_2}
=(\mathrm{F}_{4(-20)})_{E_2}=\mathrm{Spin}^0(8,1).
\end{gather*}
Then
\[\mathcal{H} \simeq
\mathrm{F}_{4(-20)}/\mathrm{Spin}(9),\quad
\mathcal{H'} \simeq
\mathrm{F}_{4(-20)}/\mathrm{Spin}^0(8,1).\] 
We denote
$\mathrm{D}_4:=(\mathrm{F}_{4(-20)})_{E_1,E_2,E_3}(\subset K)$.
From \cite[Lemma~3.2(1) and
Proposition~2.6(1)]{Na2012},
$\mathrm{D}_4$
is a connected two-hold covering group of $\mathrm{SO}(8)$,
and set $\mathrm{Spin}(8):=\mathrm{D}_4$.
We denote the Lie algebras  
$\mathfrak{f}_{4(-20)}:=Lie(\mathrm{F}_{4(-20)})$
and
$\mathfrak{d}_4:=Lie(\mathrm{D}_4)
=\{D \in \mathfrak{f}_{4(-20)}|
~D E_i=0,~i=1,2,3\}$, respectively.
From \cite[Lemma~3.9]{Na2012},
$\mathfrak{f}_{4(-20)}$
has the decomposition
\[\mathfrak{f}_{4(-20)}=
\mathfrak{d}_4\oplus \tilde{\mathfrak{u}}_1^1
\oplus \tilde{\mathfrak{u}}_2^1
\oplus \tilde{\mathfrak{u}}_3^1\quad
\text{where}
~\tilde{\mathfrak{u}}_i^1:=\{
\tilde{A}_i^1(a)|~a\in {\bf O}\}\]
(see  \cite[\S 3]{Na2012}).
The differential
$d\tilde{\sigma}$ of $\tilde{\sigma}$
at the identity
is often denoted by $\tilde{\sigma}$.
From \cite[Lemma~7.2(2)]{Na2012}, 
$d\tilde{\sigma}$
is a Cartan involution
with a Cartan decomposition
$\mathfrak{f}_{4(-20)}=\mathfrak{k}\oplus \mathfrak{p}$.
We denote $a_t:=\exp(t\tilde{A}_3^1(1))$
with $t \in \mathbb{R}$, 
the one-parameter subgroup
$A:=\{a_t|~t\in\mathbb{R}\}$,
the Lie algebra
$\mathfrak{a}:=\{t\tilde{A}_3^1(1)|~t\in\mathbb{R}\}$
of $A$,
and the linear functional $\alpha$
on $\mathfrak{a}$
such that $\alpha(\tilde{A}_3^1(1))=1$.
Set
the centralizer $M:=\{k\in K|
~k\tilde{A}_3^1(1)k^{-1}=\tilde{A}_3^1(1)\}$ 
of $\mathfrak{a}$ 
of $K$,
and its Lie subalgebra 
$\mathfrak{m}:
=\{\phi \in\mathfrak{k}|
~[\phi,\tilde{A}_3^1(1)]=0\}$.
Then
\[\tag{\ref{itd}.10}\label{itd010}
ma=am\quad \text{for~all}~m \in M~\text{and}~
a \in A.\]
From \cite[Lemma~3.2(2) and
Proposition~2.6(2)]{Na2012},
$(\mathrm{F}_{4(-20)})_{E_1,E_2,E_3,F_3^1(1)}$
is a connected two-hold covering group of $\mathrm{SO}(7)$,
and set $\mathrm{Spin}(7)
:=(\mathrm{F}_{4(-20)})_{E_1,E_2,E_3,F_3^1(1)}$.
By \cite[Proposition~7.4]{Na2012},
\[
\tag{\ref{itd}.11}\label{itd011}
M=\mathrm{Spin}(7)=
(\mathrm{F}_{4(-20)})_{E_1,E_2,E_3,F_3^1(1)}=
(\mathrm{F}_{4(-20)})_{E_j,F_3^1(1)}
\]
with $j\in\{1,2\}$.
In particular, $M$ is connected.
From \cite[Lemma~7.5]{Na2012}, 
$\mathfrak{a}$ is a maximal abelian subspace 
of $\mathfrak{p}$ with 
the following root space decomposition
of $(\mathfrak{f}_{4(-20)},\mathfrak{a})$:
\[
\tag{\ref{itd}.12}\label{itd012}
\mathfrak{f}_{4(-20)}=\mathfrak{g}_{-2\alpha}\oplus
\mathfrak{g}_{-\alpha}\oplus
\mathfrak{a}\oplus \mathfrak{m}\oplus
\mathfrak{g}_{\alpha}\oplus \mathfrak{g}_{2\alpha},
\]
the set of roots 
$\Sigma=\{\pm\alpha,\pm2\alpha\}$,
and
$\mathfrak{n}^\pm
=\mathfrak{g}_{\pm\alpha}\oplus
 \mathfrak{g}_{\pm2\alpha}~(resp)$.
Then 
$\mathfrak{g}_{\alpha}$ 
$(resp.~\mathfrak{g}_{-\alpha})$ is parameterized by 
the octonions ${\bf O}$:
\[
\tag{\ref{itd}.13}\label{itd013}
\mathfrak{g}_{\alpha}=\{\mathcal{G}_1(x)|~x\in{\bf O}\}
\quad (resp.~
\mathfrak{g}_{-\alpha}=\{\mathcal{G}_{-1}(x)|~x\in{\bf O}\})
\]
where $\mathcal{G}_{\pm 1}(x):=
\tilde{A}_1^1(x)+\tilde{A}_2^1(\mp \overline{x})
~(resp)$
and $\mathfrak{g}_{2\alpha}$ 
$(resp.~\mathfrak{g}_{-2\alpha})$
is parameterized by 
the vector parts $\mathrm{Im}{\bf O}:=\{\sum_{i=1}^7r_i e_i|~r_i \in
\mathbb{R}\}$ of octonions:
\[
\tag{\ref{itd}.14}\label{itd014}
\mathfrak{g}_{2\alpha}=\{\mathcal{G}_2(p)|
~p\in\mathrm{Im}{\bf O}\}
\quad (resp.~
\mathfrak{g}_{-2\alpha}=\{\mathcal{G}_{-2}(p)|
~p\in\mathrm{Im}{\bf O}\})
\]
where $\mathcal{G}_{\pm 2}(p):=
\tilde{A}_3^1(\mp p)-\delta(p)
~(resp)$
and $\delta(p)
 \in \mathfrak{m}\subset
\mathfrak{d}_4$ (see \cite[\S 7]{Na2012}).
Set
$N^\pm:=\exp \mathfrak{n}^\pm
=\{\exp(\mathcal{G}_{\pm 1}(x)+\mathcal{G}_{\pm 2}(p))|~
x\in{\bf O},~p\in\mathrm{Im}{\bf O}\}~(resp)$.
Because of 
$[\mathcal{G}_{\pm 1}(x), \mathcal{G}_{\pm 2}(p)]=0
~(resp)$,
\begin{align*}
\tag{\ref{itd}.15}\label{itd015}
\exp\mathcal{G}_{\pm 2}(p)\exp\mathcal{G}_{\pm 1}(x)
&=
\exp(\mathcal{G}_{\pm 1}(x)+\mathcal{G}_{\pm 2}(p))\\
&=\exp\mathcal{G}_{\pm 1}(x)\exp\mathcal{G}_{\pm 2}(p)
\quad (resp).
\end{align*}
By \cite[Lemma~7.1]{Na2012},
for any $D \in \mathfrak{d}_4$ and $a \in {\bf O}$,
\[
\tag{\ref{itd}.16}\label{itd016}
\left
\{\begin{array}{rlll}
{\rm (i)}&d\tilde{\sigma}_iD=D,&
{\rm (ii)}&d\tilde{\sigma}_i\tilde{A}_i^1(a)
=\tilde{A}_i^1(a),\\
\smallskip
{\rm (iii)}&
d\tilde{\sigma}_i\tilde{A}_j^1(a)=-\tilde{A}_j^1(a)
&\multicolumn{2}{l}{{\rm for}~j=i+1,i+2}
\end{array}\right.\]
where indexes $i,i+1,i+2,j$ are counted
modulo $3$.
Then we get
\begin{gather*}
\tag{\ref{itd}.17}\label{itd017}
d\tilde{\sigma}\mathcal{G}_{\pm 1}(x)=\mathcal{G}_{\mp 1}(x),
\quad
d\tilde{\sigma}\mathcal{G}_{\pm 2}(p)=\mathcal{G}_{\mp 2}(p)
\quad (resp),\\
\tag{\ref{itd}.18}\label{itd018}
\tilde{\sigma}\exp(\mathcal{G}_{\pm 1}(x)+\mathcal{G}_{\pm 2}(p))
=\exp(\mathcal{G}_{\mp 1}(x)+\mathcal{G}_{\mp 2}(p))
\quad (resp)
\end{gather*}
with $x\in{\bf O}$ and $p \in \mathrm{Im}{\bf O}$.
Especially, 
$\tilde{\sigma}(N^{\pm})=N^{\mp}~(resp)$.
By \cite[Corollary~8.9]{Na2012},
\[\tag{\ref{itd}.19}\label{itd019}
(\mathrm{F}_{4(-20)})_{P^-}=N^+M=MN^+.
\]
Then from (\ref{itd007}), 
\[\mathcal{N}_1^-
\simeq \mathrm{F}_{4(-20)}/MN^+.\]
Fix the Cartan involution $\theta:=d\tilde{\sigma}$ 
and set
$\epsilon(\alpha)=\epsilon(-\alpha):=-1$ and
$\epsilon(2\alpha)=\epsilon(-2\alpha):=1$ on $\Sigma$.
Then $\epsilon$ satisfies conditions (i) and (ii)
of the signature of roots,
and we consider the involutive automorphism $\theta_{\epsilon}$.
We use same notations
$\mathfrak{k}_{\epsilon},$
$(K_{\epsilon})^0,$
$K_{\epsilon},$
$M^*,$ $M_{\epsilon}^*,$
$W$, and $W_{\epsilon}$ corresponding to
notations of given for
general $G$, respectively.
\medskip

\begin{proposition}\label{itd-02}
{\rm (1)}
$\theta_{\epsilon}=d\tilde{\sigma}_2$
on $\mathfrak{f}_{4(-20)}$.
\smallskip

\noindent
{\rm (2)}
$\theta_{\epsilon}$
can be lifted on the group $\mathrm{F}_{4(-20)}$
as $\tilde{\sigma}_2$ and 
\[
\tag{\ref{itd}.20}\label{itd020}
K_{\epsilon}
=(\mathrm{F}_{4(-20)})^{\tilde{\sigma}_2}
=(\mathrm{F}_{4(-20)})_{E_2}
=\mathrm{Spin}^0(8,1).\]
\end{proposition}
\begin{proof}
Since $M\subset 
\mathrm{D}_4$ by (\ref{itd011}),
$\mathfrak{m}\subset \mathfrak{d}_4$.
Let $t \in \mathbb{R}$,
$D \in \mathfrak{m}$,
$x \in {\bf O}$, and $p \in \mathrm{Im}{\bf O}$.
Then using
(\ref{itd016}),  (\ref{itd017}) and
the definition of $\epsilon$,
\begin{gather*}
d\tilde{\sigma}_2(t\tilde{A}_3^1(1)+D)
=-t\tilde{A}_3^1(1)+D
=\theta(t\tilde{A}_3^1(1)+D)
=\theta_{\epsilon}(t\tilde{A}_3^1(1)+D),\\
d\tilde{\sigma}_2\mathcal{G}_{\pm 1}(x)
=-\mathcal{G}_{\mp 1}(x)
=\epsilon(\pm \alpha)\theta\mathcal{G}_{\pm 1}(x)
=\theta_{\epsilon}\mathcal{G}_{\pm 1}(x),\\
d\tilde{\sigma}_2\mathcal{G}_{\pm 2}(p)
=\mathcal{G}_{\mp 2}(p)
=\epsilon(\pm 2\alpha)\theta
\mathcal{G}_{\pm 2}(p)
=\theta_{\epsilon}
\mathcal{G}_{\pm 2}(p).
\end{gather*}
Thus it follows
from
(\ref{itd012}), (\ref{itd013}),
and (\ref{itd014})
that $d\tilde{\sigma}_2=\theta_{\epsilon}$
on $\mathfrak{f}_{4(-20)}$.
Then
$\theta_{\epsilon}$
can be lifted on $\mathrm{F}_{4(-20)}$ 
as $\tilde{\sigma}_2$.
From (\ref{itd009}),  we see
$(K_{\epsilon})^0=\mathrm{Spin}^0(8,1)
=(\mathrm{F}_{4(-20)})^{\tilde{\sigma}_2}=(\mathrm{F}_{4(-20)})_{E_2}$,
and 
$M
\subset (\mathrm{F}_{4(-20)})_{E_2}$
by (\ref{itd011}).
Therefore 
$K_{\epsilon}=(K_{\epsilon})^0M
=(\mathrm{F}_{4(-20)})_{E_2}$.
\end{proof}
\medskip

\begin{proposition}\label{itd-03}
{\rm (1)} 
$M^*=M\coprod \sigma M$.
Especially, 
\[
W=\{M,\sigma M\}
\cong \{1,\sigma\}
\cong \mathbb{Z}_2.\]
{\rm (2)}
$M_{\epsilon}^*=M^*=M\coprod \sigma M$.
Especially, 
\[
W_{\epsilon}=\{M,\sigma M\}\cong\{1,\sigma\}
\cong\mathbb{Z}_2,
\quad [W:W_{\epsilon}]=1.\]
\end{proposition}
\begin{proof}
(1) Fix $k \in M^*$.
Then $k\tilde{A}_3^1(1)k^{-1}=\tilde{A}_3^1(t)$
for some $t \in \mathbb{R}$.
We set $B$ as the Killing form  of
$\mathfrak{f}_{4(-20)}$,
and a negative definite inner product
$B_{\tilde{\sigma}}(\phi,\phi')
:=B(\phi,\tilde{\sigma}\phi')$ for $\phi,\phi'\in 
\mathfrak{f}_{4(-20)}$.
Then
$B_{\tilde{\sigma}}
(\tilde{A}_3^1(1),\tilde{A}_3^1(1))
=B_{\tilde{\sigma}}(k\tilde{A}_3^1(1)k^{-1},
k\tilde{A}_3^1(1)k^{-1})
=t^2 B_{\tilde{\sigma}}
(\tilde{A}_3^1(1),\tilde{A}_3^1(1))$.
Thus $t=\pm 1$,
so that $M^*=
\{k \in K|~k\tilde{A}_3^1(1)k^{-1}=\tilde{A}_3^1(\pm 1)\}$.
Put $L=\{k \in K|~k\tilde{A}_3^1(1)k^{-1}=\tilde{A}_3^1(-1)\}$.
Then $M^*=M\coprod L$.
Now,
$\sigma \in (\mathrm{F}_{4(-20)})_{E_1}=K$
by (\ref{itd008}),
and  
$\sigma \tilde{A}_3^1(1)
\sigma^{-1}=\tilde{\sigma}\tilde{A}_3^1(1)
=\tilde{A}_3^1(-1)$
by (\ref{itd016}). 
Therefor $\sigma \in M^*$, and
since  
$\sigma k \in M$ for all $k \in L$,
we get
$L
=\sigma M$.
Hence (1) follows.

(2)
Because of $\sigma E_2=E_2$
and (\ref{itd020}), we see
$\sigma\in (\mathrm{F}_{4(-20)})_{E_2}=K_{\epsilon}$.
Then 
$\sigma\in K_{\epsilon}\cap M^*=M_{\epsilon}^*$.
Therefore, because
$M$ is a subgroup of $M_{\epsilon}^*$ and (1),
$M^*=M\coprod \sigma M\subset M_{\epsilon}^*
\subset M^*$,
and so (2) follows.
\end{proof}
\medskip

From $[W:W_{\epsilon}]=1$ and Proposition~\ref{itd-01},
the submanifold $K_{\epsilon}AN^+$ is open dense
in $\mathrm{F}_{4(-20)}$,
and for any
$g \in K_{\epsilon}AN^+$,
there exist unique elements
$k_{\epsilon}(g) \in K_{\epsilon}$,
$H_{\epsilon}(g) \in \mathfrak{a}$, and
$n_{\epsilon}(g) \in N^+$
such that
\[g=k_{\epsilon}(g)\exp(H_{\epsilon}(g))n_{\epsilon}(g).\]
However, this fact
will be actually shown  in this article.

For $x \in {\bf O}$,
we denote 
$Q^+(x):=h^1(0,0,0;x,\overline{x},0)$
and $Q^-(x):=h^1(0,0,0;x,-\overline{x},0)$.
We will prove the following main-theorem
in \S 11.

\begin{main-theorem}\label{itd-04}
{\rm (The explicit Iwasawa decomposition 
of $\mathrm{F}_{4(-20)}$).}
For any $g\in \mathrm{F}_{4(-20)}$,
there exist unique $k(g) \in K$, $H(g) \in \mathfrak{a}$,
and $n_I(g) \in N^+$ such that
\[g=k (\exp H(g)) n_I(g)\]
where
\begin{align*}
\tag{i}
H(g)&=2^{-1}\log\bigl(-(gP^-|E_1)\bigr)
\tilde{A}_3^1(1)\in \mathfrak{a},\\
\tag{ii}
n_I(g)&=\exp\Biggl(\mathcal{G}_1\left(
2^{-1}\bigl(\sum{}_{i=0}^7 (g Q^+(e_i)|E_1)e_i
\bigr)
/(gP^-|E_1)\right)\\
&+\mathcal{G}_2\left(
-2^{-1}\bigl(\sum{}_{i=1}^7 (g F_3^1(e_i)|E_1)
e_i\bigr)/(gP^-|E_1)\right)
\Biggr)\in N^+,\\
\tag{iii}
k(g)&=g n_I(g)^{-1}\exp (-H(g))\in K.
\end{align*}
\end{main-theorem}
\medskip

We define the equivalence relation
$\sim$
on $\mathcal{N}_1^-$ by
\[X\sim Y\quad\quad 
\stackrel{{\rm def}}{\Leftrightarrow}
\quad\quad Y=rX\quad\text{for~some~}r>0\]
where 
$X,Y\in\mathcal{N}_1^-$. 
We denote the quotient set 
\[\mathcal{F}:=\mathcal{N}_1^-/\sim,\]
and the equivalence class of $X\in\mathcal{N}_1^-$
by $[X]$.
From (\ref{itd007}),
$\mathrm{F}_{4(-20)}$ acts on $\mathcal{F}$:
\[g[X]:=[gX]\quad\quad\text{for}~g
\in \mathrm{F}_{4(-20)}
~~\text{and}~~ X\in\mathcal{N}_1^-.\]
We will prove the following theorem
in \S 11.

\begin{theorem}\label{itd-05}
\begin{gather*}
\tag{1}
(\mathrm{F}_{4(-20)})_{[P^-]}=MAN^+.\\
\tag{2}
\mathrm{F}_{4(-20)}/MAN^+\simeq \mathcal{F}.\\
\tag{3}
\mathcal{F}=K \cdot [P^-].\\
\tag{4}
\mathcal{F} \simeq \mathrm{Spin}(9)/\mathrm{Spin}(7).
\end{gather*}
\end{theorem}
\medskip

Since $r X \in \mathcal{N}_1^-$ 
for all $r>0$ and $X \in \mathcal{N}_1^-$,
$\mathcal{N}_1^-$ 
is a cone 
in $\mathcal{J}^1$. Setting
$-\mathcal{N}_1^+:=\{-X|~X \in 
\mathcal{N}_1^+\}$, we see that
$\mathcal{N}_1^-=-\mathcal{N}_1^+$
from the definitions of $\mathcal{N}_1^+$
and $\mathcal{N}_1^-$,
and that $\sigma P^-=-P^+$.
And noting that 
$\mathcal{F}=\mathcal{N}_1^-/\sim$,
(\ref{itd003}),
and 
$\mathcal{F}\simeq S^{15}$
(see
Proposition~\ref{id-02}),
we draw the following figure.

\includegraphics[width= 8 cm]{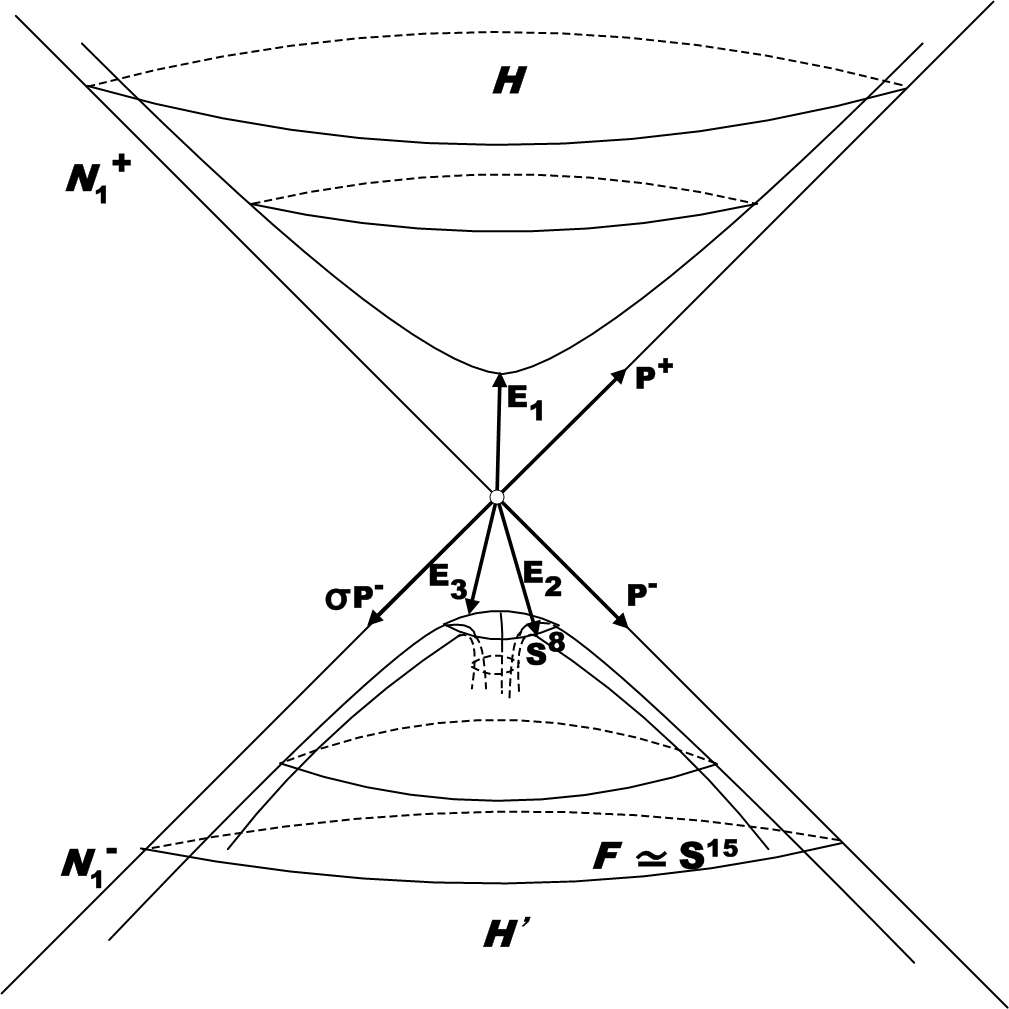}

We will prove the following main-theorem in
\S 12.

\begin{main-theorem}\label{itd-06}
{\rm (The explicit $K_{\epsilon}-$Iwasawa decomposition 
of $\mathrm{F}_{4(-20)}$)}.
\begin{align*}
K_{\epsilon}AN^+&=\{g\in \mathrm{F}_{4(-20)}|
~(gP^-|E_2)\ne 0\}\\
&=\{g\in \mathrm{F}_{4(-20)}|~(gP^-|E_2)> 0\}.
\end{align*}
Furthermore, the submanifold
$K_{\epsilon}AN^+$ is open dense
in $\mathrm{F}_{4(-20)}$.

For any $g\in K_{\epsilon}AN^+$,
there exist unique $k_{\epsilon}(g) \in K_{\epsilon}$, 
$H_{\epsilon}(g) \in \mathfrak{a}$,
and $n_{\epsilon}(g) \in N^+$ such that
\[g=k_{\epsilon} (\exp H_{\epsilon}(g)) n_{\epsilon}(g)\]
where
\begin{gather*}
\tag{i}
H_{\epsilon}(g)=2^{-1}\log\bigl((gP^-|E_2)\bigr)
\tilde{A}_3^1(1)\in\mathfrak{a},\\
\tag{ii}
n_{\epsilon}(g)
=\exp\Biggl(\mathcal{G}_1\left(
2^{-1} \bigl(\sum{}_{i=0}^7 (g Q^+(e_i)|E_2)e_i
\bigr)
/(gP^-|E_2)\right)\\
+\mathcal{G}_2\left(
-2^{-1} \bigl(\sum{}_{i=1}^7 (g F_3^1(e_i)|E_2)
e_i\bigr)/(gP^-|E_2)\right)
\Biggr)\in N^+,\\
\tag{iii}
k_{\epsilon}(g)=gn_{\epsilon}(g)^{-1}\exp (-H_{\epsilon}(g)).
\end{gather*}
\end{main-theorem}
\medskip

We denote the elements 
$P_{12}^-,P_{13}^-\in \mathcal{J}^1$ by
$P_{12}^-:=h^1(-1,1,0;0,0,1)
=P^-$ and
$P_{13}^-:=h^1(-1,0,1;0,1,0)$,
respectively.
We will prove
the following theorems in \S 13.

\begin{theorem}\label{itd-07}
$\mathcal{F}$ 
decomposes into the following two $K_{\epsilon}$-orbits:
\[
\mathcal{F}=\coprod{}_{i=2}^3 K_{\epsilon} \cdot [P^-_{1i}]
\]
where
\begin{align*}
K_{\epsilon} \cdot [P^-_{12}]
&=\{[X]\in\mathcal{F}|~(X|E_2)\ne 0\}
=\{[X]\in\mathcal{F}|~(X|E_2)> 0\},\\
K_{\epsilon} \cdot [P^-_{13}]
&=\{[X]\in\mathcal{F}|~(X|E_2)=0\}.
\end{align*}
\end{theorem}
\medskip

\begin{main-theorem}\label{itd-08}
{\rm (The explicit Matsuki decomposition 
of $\mathrm{F}_{4(-20)}$)}.
\[
\mathrm{F}_{4(-20)}=K_{\epsilon}MAN^+\coprod 
K_{\epsilon}\exp\left(-2^{-1}\pi\tilde{A}^1_1
(1)\right)MAN^+
\]
where $K_{\epsilon}MAN^+=K_{\epsilon}AN^+$ and
\[
K_{\epsilon}
\exp\left(-2^{-1}\pi\tilde{A}^1_1(1)\right)MAN^+
=\{g\in \mathrm{F}_{4(-20)}|~(gP^-|E_2)= 0\}.
\]
\end{main-theorem}

Theorems~\ref{itd-07} and \ref{itd-08}
are special cases of 
general theory
\cite[Theorems 1-Corollary and 3]{Mt1979}.
\medskip

Since the Bruhat decomposition is associated with
the $N^-$-orbits on $\mathrm{F}_{4(-20)}/MAN^+$,
we will show the following theorem in \S 14.

\begin{theorem}\label{itd-09}
$\mathcal{F}$
decomposes into the following two $N^-$-orbits:
\[
\mathcal{F}=N^- \cdot [P^-]
\coprod N^- \cdot [\sigma P^-]\]
where
\begin{align*}
N^- \cdot [P^-]
&=\{[X]\in\mathcal{F}|~(X|\sigma P^-)> 0\}
=\{[X]\in\mathcal{F}|~(X|\sigma P^-)\ne 0\},\\
N^- \cdot [\sigma P^-]
&=\{[X]\in\mathcal{F}|~(X|\sigma P^-)=0\}
=\{[\sigma P^-]\}.
\end{align*}
\end{theorem}

We will prove the following 
main-theorem
in \S 14.

\begin{main-theorem}\label{itd-10}
{\rm (1) (The explicit Bruhat decomposition 
of $\mathrm{F}_{4(-20)}$)}.
\[
\mathrm{F}_{4(-20)}
=N^-MAN^+\coprod \sigma MAN^+
\]
where
\begin{align*}
N^-MAN^+
&=\{g\in \mathrm{F}_{4(-20)}|~(g P^-|\sigma P^-)\ne 0\}\\
&=\{g\in \mathrm{F}_{4(-20)}|~(g P^-|\sigma P^-)> 0\},\\
\sigma MAN^+&=N^-\sigma MAN^+\\
&=\{g\in \mathrm{F}_{4(-20)}|~(g P^-|\sigma P^-)= 0\}\\
&=\{g\in \mathrm{F}_{4(-20)}|~g [P^-]=[\sigma P^-]\}.
\end{align*}
Furthermore,
the submanifold $N^-MAN^+$
is open dense in $\mathrm{F}_{4(-20)}$.
\smallskip

\noindent
{\rm (2) (The explicit Gauss decomposition 
of $\mathrm{F}_{4(-20)}$)}.

\noindent
For any $g\in N^-MAN^+$, 
there exist unique $n^-_G(g) \in N^-$, 
$m_G(g) \in M$,
$a_G(g) \in A$,
and $n^+_G(g) \in N^+$ such that
\[g=n^-_G(g)m_G(g)a_G(g)n^+_G(g)\]
where 
\begin{gather*}
\tag{i}
a_G(g)=\exp\left(2^{-1}\log
\bigl(4^{-1}(g P^-|\sigma P^-)\bigr)
\tilde{A}_3^1(1)\right)\in A,\\
\tag{ii}
n^-_G(g)
=\exp\Biggl(\mathcal{G}_{-1}\left(
-2^{-1} \bigl(\sum{}_{i=0}^7 (Q^-(e_i)|g P^-)e_i
\bigr)
/(gP^-|\sigma P^-)\right)\\
+\mathcal{G}_{-2}\left(
-2^{-1} \bigl(\sum{}_{i=1}^7 (F_3^1(e_i)|gP^-)
e_i\bigr)/(gP^-|\sigma P^-)\right)
\Biggr)\in N^-,\\
\tag{iii}
n^+_G(g)=n_I(n^-_G(g)^{-1}g)
\in N^+,\\
\tag{iv}
m_G(g)= n^-_G(g)^{-1} g n^+_G(g)^{-1} a_G(g)^{-1}
\in M.
\end{gather*}
Here $n_I:\mathrm{F}_{4(-20)} \to N^+$
is the map used in the Iwasawa decomposition.
\end{main-theorem}
\medskip

\begin{remark}\label{itd-11}
\begin{rm}
In Main Theorems~\ref{itd-04},~\ref{itd-06},~\ref{itd-08},
and \ref{itd-10}, it appears that
the Iwasawa decomposition,
the $K_{\epsilon}-$Iwasawa decomposition,
the Matsuki decomposition,
and 
the Bruhat and Gauss decompositions
of $\mathrm{F}_{4(-20)}$
can be explicitly described 
by using the geometric quantities 
$(g P^-|E_1)$, $(g P^-|E_2)$,
and $(g P^-|\sigma P^-)$
with $g \in \mathrm{F}_{4(-20)}$.
\end{rm}
\end{remark}

\begin{remark}\label{itd-12}
\begin{rm}
The Iwasawa decomposition
of  the exceptional Lie group $\mathrm{F}_{4(-20)}$ 
has been studied 
by R.~Takahashi \cite[Theorem 1]{Tr1979}.
He showed that $AN^+$ transitively and freely
acts on the hyperbolic plane $\mathcal{H}=\mathrm{F}_{4(-20)}/K$.
Thereby, he gave the existence and uniqueness of the factors
of the Iwasawa decomposition for $\mathrm{F}_{4(-20)}$.
In Main-Theorem \ref{itd-04},
we give explicit formulas of $H(g)$ and $n_I(g)$.
\end{rm}
\end{remark}

\section{Preliminaries.}\label{prl2}

If $X=\sum_{i=1}^3(\xi_i E_i+F_i^1(x_i))\in\mathcal{J}^1$,
then we denote
$(X)_{E_i}:=\xi_i \in \mathbb{R}$ and
$(X)_{F_i^1}=x_i \in {\bf O}$.
Set 
$F_3^1({\rm Im}{\bf O})
:=\{F_3^1(p)|~p\in{\rm Im}{\bf O}\}$,
$Q^+({\bf O}):=\{Q^+(x)|~x\in{\bf O}\}$,
and
$Q^-({\bf O}):=\{Q^-(x)|~x\in{\bf O}\}$
in $\mathcal{J}^1$.
Then
\begin{align*}
\tag{\ref{prl2}.1}\label{prl2001}
\mathcal{J}^1&=\mathbb{R}(-E_1+E_2)
\oplus \mathbb{R}P^-\oplus \mathbb{R}E\oplus
\mathbb{R}E_3\oplus F_3^1({\rm Im}{\bf O})\\
&\quad \oplus Q^+({\bf O})\oplus Q^-({\bf O}).
\end{align*}
So, for any $X\in\mathcal{J}^1$, 
we can uniquely write 
\begin{align*}
X&=r(-E_1+E_2)+sP^-+u E+v E_3
+F_3^1(p)+Q^+(x)+Q^-(y)\\
&=
\begin{pmatrix}
-r-s+u & \sqrt{-1}(s+p) & \sqrt{-1}(x-y)\\
\sqrt{-1}(s-p) & r+s+u & x+y\\
\sqrt{-1}(\overline{x}-\overline{y}) 
& \overline{x}+\overline{y}
& u+v
\end{pmatrix}
\end{align*}
with $r,s,u,v\in\mathbb{R},$
$p\in{\rm Im}{\bf O}$, and $x,y\in{\bf O}$,
and set
\begin{align*}
\{X\}_{-E_1+E_2}&:=r,&\{X\}_{P^-}&:=s,&\{X\}_{E}&:=u,
&\{X\}_{E_3}&:=v,\\
\{X\}_{{\rm Im}F_3^1}&:=p,
&\{X\}_{Q^+}&:=x,
&\{X\}_{Q^-}&:=y.
\end{align*}
\begin{lemma}\label{prl2-01}
\begin{gather*}
\tag{1}
\{X\}_{-E_1+E_2}=2^{-1}(P^-|X).\\
\tag{2}
\{X\}_{Q^-}
=2^{-1}((X)_{F_1^1}-\overline{(X)_{F_2^1}})
=4^{-1}\sum{}_{i=0}^7(Q^+(e_i)|X)e_i.\\
\tag{3}
\{X\}_{\mathrm{Im}F_3^1}={\rm Im}((X)_{F_3^1})
=-2^{-1}\sum{}_{i=1}^7(F_3^1(e_i)|X)e_i.
\end{gather*}
\end{lemma}
\begin{proof} Let $X=r(-E_1+E_2)+sP^-+uE+vE_3
+F_3^1(p)+Q^+(x)+Q^-(y)$
with $r,s,u,v\in\mathbb{R},$
$p\in{\rm Im}{\bf O}$, and $x,y\in{\bf O}$.
Then
$(P^-|X)=2r$, and so (1) follows.
Because of 
$(X)_{F_1^1}=x+y$ and $(X)_{F_2^1}=\overline{x}-\overline{y}$,
$\{X\}_{Q^-}=y=2^{-1}((X)_{F_1^1}-\overline{(X)_{F_2^1}})$.
Now, set $(X)_{F_1^1}=\sum_{i=0}^7p_ie_i$ and
$(X)_{F_2^1}=\sum_{i=0}^7q_ie_i$
with $p_i,q_i\in \mathbb{R}$.
From (\ref{itd001}),
$p_i=2^{-1}(F_1^1(e_i)|X)$ and $q_i=-2^{-1}(F_2^1(e_i)|X)$.
Then
$(X)_{F_1^1}-\overline{(X)_{F_2^1}}
=2^{-1}\sum{}_{i=0}^7(Q^+(e_i)|X)e_i$,
and so (2) follows.
Last, obviously
$\{X\}_{\mathrm{Im}F_3^1}=p={\rm Im}((X)_{F_3^1})$.
Set $(X)_{F_3^1}=\sum_{i=0}^7r_ie_i$
with $r_i \in \mathbb{R}$.
From (\ref{itd001}),
$r_i=-2^{-1}(F_3^1(e_i)|X)$,
and so (3) follows.
\end{proof}

We denote $\mathcal{J}^1(2;{\bf K})
:=\{\xi_1E_1+\xi_2E_2+F_3^1(x)~|~
\xi_i\in\mathbb{R},~x\in{\bf K}\}$
with ${\bf K}={\bf O}$ or $\mathbb{R}.$
\begin{lemma}\label{prl2-02}
\begin{align*}
\tag{1}
\mathcal{J}^1
=&\mathcal{J}^1(2;{\bf O})\oplus
\mathbb{R}E_3\oplus
Q^+({\bf O})\oplus
Q^-({\bf O}).\\
\tag{2}
\mathcal{J}^1(2;{\bf O})
=&\mathbb{R}(-E_1+E_2)\oplus
\mathbb{R}P^-\oplus \mathbb{R}(E-E_3)
\oplus F_3^1({\rm Im}{\bf O}).\\
\tag{3}
\mathcal{J}^1(2;\mathbb{R})
=&\mathbb{R}(-E_1+E_2)\oplus
\mathbb{R}P^-\oplus \mathbb{R}(E-E_3).
\end{align*}
\end{lemma}
\medskip

Let $p,q\in{\rm Im}{\bf O}$ and $x,y\in{\bf O}$.
From \cite[Lemma~7.11]{Na2012},
\[
\tag{\ref{prl2}.2}\label{prl2002}
\left\{\begin{array}{rl}
{\rm (i)}& \exp\mathcal{G}_2(p)(-E_1+E_2)
=(-E_1+E_2)+F_3^1(-2 p)+2(p|p)P^-,\\
\smallskip
{\rm (ii)}
& \exp\mathcal{G}_2(p)P^-
=P^-,\quad
{\rm (iii)}~~~\exp\mathcal{G}_2(p)E=E,
\\
\smallskip
{\rm (iv)}&\exp\mathcal{G}_2(p)E_3=E_3,
\\
\smallskip
{\rm (v)}&
\exp\mathcal{G}_2(p)F_3^1(q)=F_3^1(q)-2(p|q)P^-,\\
\smallskip
{\rm (vi)}& 
\exp\mathcal{G}_2(p)Q^+(y)
=Q^+(y),\\
\smallskip
{\rm (vii)}&
\exp\mathcal{G}_2(p)Q^-(y)=Q^-(y)
+Q^+(-2 p y),
\end{array}\right.
\]
\[
\left\{\begin{array}{rl}
\tag{\ref{prl2}.3}\label{prl2003}
{\rm (i)}&
\exp\mathcal{G}_1(x)(-E_1+E_2)
  =(-E_1+E_2)+Q^-(-x)\\
\smallskip
&-(x|x)(E-3E_3)+Q^+\left((x|x)x\right)
+2^{-1}(x|x)^2P^-,\\
\smallskip
{\rm (ii)}&\exp\mathcal{G}_1(x)P^-=P^-,\quad
{\rm (iii)}~~~\exp\mathcal{G}_1(x)E=E,\\
\smallskip
{\rm (iv)}& \exp\mathcal{G}_1(x)E_3
=E_3+Q^+(x)+(x|x)P^-,\\
\smallskip
{\rm (v)}&
\exp\mathcal{G}_1(x)F_3^1(q)=F_3^1(q)+Q^+(- q x),\\
\smallskip
{\rm (vi)}&
\exp\mathcal{G}_1(x)Q^+(y)=Q^+(y)+2(x|y)P^-,\\
\smallskip
{\rm (vii)}&
\exp\mathcal{G}_1(x)Q^-(y)=Q^-(y)
+2(x|y)(E-3E_3)
+F_3^1\left(2{\rm Im}(x\overline{y})\right)\\
&
+Q^+\left(-3(x|y)x-{\rm Im}(x\overline{y})x
\right)-2(x|y)(x|x)P^-.
\end{array}\right.
\]
We denote the subset $\mathfrak{R}_1$ of $\mathcal{J}^1$
by $\mathcal{R}_1:=\{X \in \mathcal{J}^1|~
X^{\times 2}=0,~X \ne 0\}$,
and call $\mathfrak{R}_1$ the {\it set of  rank 1}.
Then 
$\mathcal{R}_1$ contains
the exceptional  hyperbolic planes $\mathcal{H}$,
$\mathcal{H'}$
and the exceptional null cones $\mathcal{N}_1^+$,
$\mathcal{N}_1^-$.
Since the action of $F_{4(-20)}$ preserves
the cross product,
$F_{4(-20)}$
acts on $\mathcal{R}_1$.
For any subset $S \subset \mathcal{J}^1$
and $Z \in \mathcal{J}^1$,
we denote 
\begin{align*}
 S^{Z}_{>0}&:=\{X\in S|~(Z|X)> 0\},&
 S^{Z}_{<0}&:=\{X\in S|~(Z|X)< 0\},\\
 S^{Z}_{=0}&:=\{X\in S|~(Z|X)= 0\},&
 S^{Z}_{\ne 0}&:=\{X\in S|~(Z|X)\ne 0\}.
\end{align*}
We recall Lemma~\ref{prl2-01}.
For any $X \in (\mathcal{J}^1)^{P^-}_{\ne 0}$,
we define the elements $n_1(X) \in \exp \mathfrak{g}_{\alpha}
\subset N^+$ 
and
$n_2(X) \in \exp \mathfrak{g}_{2\alpha}
\subset N^+$
by
\begin{align*}
n_1(X)
:=&\exp \mathcal{G}_1(\{X\}_{Q^-}/\{X\}_{-E_1+E_2})\\
=&
\exp \mathcal{G}_1
\left(2^{-1}\bigl(\sum{}_{i=0}^7(Q^+(e_i)|X)e_i\bigr)
/(P^-|X)\right),\\
n_2(X):
=&\exp 
\mathcal{G}_2\left(
\{X\}_{\mathrm{Im}F_3^1}/(P^-|X)\right)\\
=&\exp \mathcal{G}_2
\left(-2^{-1}\bigl(\sum{}_{i=1}^7
(F_3^1(e_i)|X)e_i\bigr)/(P^-|X)
\right)
\end{align*}
respectively,
and $n_X:=n_1(X)n_2(X)
=n_2(X)n_1(X) \in N^+$
(see (\ref{itd015})).

\begin{lemma}\label{prl2-03}
{\rm (1)}
For any $n \in N^+$ and $X \in \mathcal{J}^1$,
$(P^-|n X)=(P^-|X)$.
Especially,
$N^+$ acts on $(\mathcal{J}^1)^{P^-}_{\ne 0}$
and $(\mathcal{R}_1)^{P^-}_{\ne 0}$,
respectively.
\smallskip

\noindent
{\rm (2)} For any $X \in (\mathcal{J}^1)^{P^-}_{\ne 0}$,
\begin{align*}
\tag{i}
&n_1(X) X\in (\mathcal{J}^1(2;{\bf O})\oplus
\mathbb{R}E_3\oplus
Q^+({\bf O})) \cap (\mathcal{J}^1)^{P^-}_{\ne 0},\\
\tag{ii}
&\{n_1(X)X\}_{{\rm Im}F_3^1}=\{X\}_{{\rm Im}F_3^1}.
\end{align*}
{\rm (3)} 
If $X\in\mathcal{J}^1(2;{\bf O})\cap
 (\mathcal{J}^1)^{P^-}_{\ne 0}$,
then 
\[
n_2(X)X\in\mathcal{J}^1(2;\mathbb{R})
\cap (\mathcal{J}^1)^{P^-}_{\ne 0}.
\]
\end{lemma}
\begin{proof}
(1) 
From (\ref{itd019}),
$(P^-|n X)=(n^{-1}P^-|X)=(P^-|X)$
and so on.

(2)  
Let
$X=r(-E_1+E_2)+sP^-+uE+vE_3+F_3^1(p)
+Q^+(x)+Q^-(y)$
for some $r, s, u, v\in\mathbb{R},$
$p\in{\rm Im}{\bf O}$, 
and $x,y\in{\bf O}$.
From Lemma~\ref{prl2-01}(1), $r\ne 0$
and put 
$n_1'=n_1(X)=\exp\mathcal{G}_1(r^{-1}y)$.
In (\ref{prl2003}),
we notice that
the equations (\ref{prl2003})(i) and (\ref{prl2003})(vii) 
have terms of
$Q^-(\cdot)$
and the other equations have not
terms of
$Q^-(\cdot)$,
and that
the equations (\ref{prl2003})(v) and (\ref{prl2003})(vii) 
have terms of
$F_3^1(\cdot)$
and the other equations have not
terms of
$F_3^1(\cdot)$.
Therefore
\begin{align*}
{\{n_1'\cdot X\}_{Q^-}}
&=\{n_1'\cdot(r(-E_1+E_2)+Q^-(y)
+(other~terms))\}_{Q^-}\\
&=-r(r^{-1}y)+y+0=0.
\end{align*}
Thus 
$\{n_1(X)X\}_{Q^-}=0$,
so that
$n_1(X) X\in
\mathcal{J}^1(2;{\bf O})\oplus
\mathbb{R}E_3\oplus
Q^+({\bf O})$.
Then  $(P^-|n_1'\cdot X)=(P^-|X)\ne 0$
by (1),
and
\begin{align*}
\{n_1'\cdot X\}_{{\rm Im}F_3^1}
&=\{n_1'\cdot(F_3^1(p)
+Q^-(y)+(other~terms))
\}_{{\rm Im}F_3^1}\\
&=p
+2{\rm Im}\left((r^{-1}y)
\overline{y}\right)+0=p
=\{X\}_{{\rm Im}F_3^1}.
\end{align*}
Hence we obtain (2).

(3) Let
$X=r(-E_1+E_2)+sP^-+u(E-E_3)+F_3^1(p)$
for some $r, s, u\in\mathbb{R}$ and
$p\in{\rm Im}{\bf O}$.
From Lemma~\ref{prl2-01}(1), $r\ne 0$
and put $n_2'=n_2(X)=\exp\mathcal{G}_2((2r)^{-1}p)$.
Using (\ref{prl2002}), we calculate that
\[
n_2' \cdot X =
r(-E_1+E_2)+(s-(2r)^{-1}(p|p))P^-+u(E-E_3). 
\]
Then 
$n_2' \cdot X \in \mathcal{J}^1(2;\mathbb{R})$,
and $(P^-|n_2' \cdot X)=(P^-|X)\ne 0$.
Hence we obtain (3).
\end{proof}
\medskip

\begin{lemma}\label{prl2-04}
\[
\left(\mathcal{J}^1(2;{\bf O})\oplus
\mathbb{R}E_3\oplus
Q^+({\bf O})\right)\cap (\mathcal{R}_1)^{P^-}_{\ne 0}
=\mathcal{J}^1(2;{\bf O})
\cap (\mathcal{R}_1)^{P^-}_{\ne 0}.\]
\end{lemma}
\begin{proof}
Obviously, $\mathcal{J}^1(2;{\bf O})
\cap (\mathcal{R}_1)^{P^-}_{\ne 0} \subset
(\mathcal{J}^1(2;{\bf O})\oplus
\mathbb{R}E_3\oplus
Q^+({\bf O}))\cap (\mathcal{R}_1)^{P^-}_{\ne 0}$.
Conversely, take $X\in
(\mathcal{J}^1(2;{\bf O})\oplus
\mathbb{R}E_3\oplus
Q^+({\bf O}))\cap (\mathcal{R}_1)^{P^-}_{\ne 0}$
and set
$X=\xi_1E_1+\xi_2E_2+\xi_3E_3
+F_1^1(x)+F_2^1(\overline{x})+F_3^1(y)$
with $\xi_i\in\mathbb{R}$ and $x,y\in{\bf O}$.
Suppose that
$\xi_3 \ne 0$.
Because of $X\in \mathfrak{P}$
and (\ref{itd002}),
\begin{align*}
{\rm (i)}~&
\xi_2\xi_3-(x|x)=(X^{\times 2})_{E_1}=0,&
{\rm (ii)}~&
\xi_3\xi_1+(x|x)=(X^{\times 2})_{E_2}=0,\\
{\rm (iii)}~&
(x|x)-\xi_3 y=(X^{\times 2})_{F_3^1}=0.
\end{align*}
From (i), (ii), and (iii),  
$X=
-\left((x|x)/\xi_3\right)E_1+\left((x|x)/\xi_3\right)E_2+
\eta E_3
+F_1^1(x)+F_2^1(\overline{x})+F_3^1\left((x|x)/\xi_3\right)$.
Then $(P^-|X)=0$, and
it contradicts with $X\in (\mathcal{R}_1)^{P^-}_{\ne 0}$.
Thus $\xi_3=0$. 
Then $(x|x)=(X^{\times 2})_{E_2}=0$, 
so that $x=0$.
Thus $X=\xi_1E_1+\xi_2E_2
+F_3^1(y) \in \mathcal{J}^1(2;{\bf O})
\cap (\mathcal{R}_1)^{P^-}_{\ne 0}$,
and so $(\mathcal{J}^1(2;{\bf O})\oplus
\mathbb{R}E_3\oplus
Q^+({\bf O}))\cap 
(\mathcal{R}_1)^{P^-}_{\ne 0}
\subset \mathcal{J}^1(2;{\bf O})
\cap (\mathcal{R}_1)^{P^-}_{\ne 0}$.
Hence the result follows.
\end{proof}
\medskip

\begin{lemma}\label{prl2-05}
For any $X \in (\mathcal{R}_1)^{P^-}_{\ne 0}$,
$n_X X
\in \mathcal{J}^1(2;\mathbb{R})\cap
(\mathcal{R}_1)^{P^-}_{\ne 0}$.
Further, 
\begin{align*}
n_X X&=
4^{-1}(2\mathrm{tr}(X)
-(P^-|X)^{-1}\mathrm{tr}(X)^2-(P^-|X))E_1\\
&\quad+4^{-1}(2\mathrm{tr}(X)+(P^-|X)^{-1}\mathrm{tr}(X)^2
+(P^-|X))E_2\\
&\quad+F_3^1\left(4^{-1}(
(P^-|X)^{-1}\mathrm{tr}(X)^2-(P^-|X)
)\right)\\
&=
2^{-1}(P^-|X)(-E_1+E_2)
+4^{-1}(
(P^-|X)^{-1}\mathrm{tr}(X)^2-(P^-|X)
)P^-\\
&\quad
+2^{-1}\mathrm{tr}(X)(E-E_3).
\end{align*}
\end{lemma}
\begin{proof}
$N^+$
acts on $(\mathcal{R}_1)^{P^-}_{\ne 0}$,
and $n_i(X) \in N^+$.
Put $X'=n_1(X)X \in  (\mathcal{R}_1)^{P^-}_{\ne 0}$.
By Lemma~\ref{prl2-03}(2),
\[X'\in(\mathcal{J}(2;{\bf O})
\oplus \mathbb{R}E_3\oplus Q^+({\bf O}))
\cap
(\mathcal{R}_1)^{P^-}_{\ne 0}\]
where $(P^-|X')=(P^-|X)\ne 0$ 
and $\{X'\}_{\mathrm{Im}F_3^1}=\{X\}_{
\mathrm{Im}F_3^1}$.
Applying Lemma~\ref{prl2-04}.
\[X' \in \mathcal{J}(2;{\bf O})\cap 
(\mathcal{R}_1)^{P^-}_{\ne 0}.\]
Applying Lemma~\ref{prl2-03}(3), 
\[n_2(X')X'\in
\mathcal{J}^1(2;\mathbb{R})\cap
(\mathcal{R}_1)^{P^-}_{\ne 0}.\]
Then, since
$\exp \mathcal{G}_2(
\{X\}_{\mathrm{Im}F_3^1}/(P^-|X))
=\exp \mathcal{G}_2(
\{X'\}_{\mathrm{Im}F_3^1}/(P^-|X'))$,
we see $n_2(X)=n_2(X')$.
Therefore
\[n_X X=n_2(X)n_1(X)X=n_2(X')X'
\in \mathcal{J}^1(2;\mathbb{R})\cap
(\mathcal{R}_1)^{P^-}_{\ne 0}.\]
Set
$n_X X=\xi_1 E_1+\xi_2 E_2+F_3^1(x) \in 
\mathcal{J}^1(2;\mathbb{R})\cap
(\mathcal{R}_1)^{P^-}_{\ne 0}$ with
$\xi_1, \xi_2, x \in \mathbb{R}$.
Then $\mathrm{tr}(X)=\xi_1+\xi_2$ and
$(0 \ne) (P^-|X)=(P^-|n_X X)
=-\xi_1+\xi_2-2 x$, so that
$\xi_1=2^{-1}\mathrm{tr}(X)-x-2^{-1}(P^-|X)$
and $\xi_2=2^{-1}\mathrm{tr}(X)+x+2^{-1}(P^-|X)$.
From
$(n_X X)^{\times 2}=0$,
$0=((n_X X)^{\times 2})_{E_3}
=\xi_1\xi_2+x^2
=4^{-1}\mathrm{tr}(X)^2-4^{-1}(P^-|X)^2-x(P^-|X)$.
Thus
$x=4^{-1}(
(P^-|X)^{-1}\mathrm{tr}(X)^2-(P^-|X)
)$,
$\xi_1=4^{-1}(2\mathrm{tr}(X)
-(P^-|X)^{-1}\mathrm{tr}(X)^2-(P^-|X))$,
and $\xi_2=
4^{-1}(2\mathrm{tr}(X)+(P^-|X)^{-1}\mathrm{tr}(X)^2
+(P^-|X))$.
Moreover, the last equation follows from
direct calculations.
\end{proof}

Let $i \in \{1,2,3\}$, $t \in \mathbb{R}$, and
$a \in {\bf O}$ with $(a|a)=1$.
From \cite[Lemma~3.10]{Na2012},
we recall
the operation of $\exp \bigl(t \tilde{A}_i^1(a)\bigr)$.
Set \[h^1(\eta_1,\eta_2,\eta_3;y_1,y_2,y_3)
:=\exp \bigl(t \tilde{A}_i^1(a)\bigr)
h^1(\xi_1,\xi_2,\xi_3;x_1,x_2,x_3)\]
with $\xi_i,\eta_i \in \mathbb{R}$ and $x_i,y_i \in
{\bf O}$.
When $i=1$,
\[\tag{\ref{prl2}.4}\label{prl2004}
\left\{
\begin{array}{ccl}
\eta_1&=&\xi_1,\\
\eta_2&=&2^{-1}((\xi_2+\xi_3)
+(\xi_2-\xi_3)\cos 2t)+(a|x_1) \sin 2t,\\
\eta_3&=&2^{-1}((\xi_2+\xi_3)
-(\xi_2-\xi_3)\cos 2t)-(a|x_1) \sin 2t,\\
y_1&=&x_1-2^{-1}(\xi_2-\xi_3) a\sin 2t
-2(a|x_1)a\sin^2 t,\\
y_2&=&x_2\cos t-\overline{x_3a} \sin t,\\
y_3&=&x_3\cos t+\overline{ax_2} \sin t
\end{array}\right.
\]
and when $i\in\{2,3\}$, 
\[\tag{\ref{prl2}.5}\label{prl2005}
\left\{
\begin{array}{ccl}
\eta_i&=&\xi_i,\\
\eta_{i+1}
&=&2^{-1}((\xi_{i+1}+\xi_{i+2})
+(\xi_{i+1}-\xi_{i+2})\cosh 2t)-(a|x_i) \sinh 2t,\\
\eta_{i+2}&=&2^{-1}((\xi_{i+1}+\xi_{i+2})
-(\xi_{i+1}-\xi_{i+2})\cosh 2t)+(a|x_i) \sinh 2t,\\
y_i&=&x_i
-2^{-1}(\xi_{i+1}-\xi_{i+2})a\sinh 2t+2(a|x_i)a\sinh^2 t.\\
y_{i+1}&=&x_{i+1}\cosh t+\overline{x_{i+2}a}\sinh t,\\
y_{i+2}&=&x_{i+2}\cosh t+\overline{ax_i} \sinh t
\end{array}\right.
\]
where indexes $i,i+1,i+2$ are
counted modulo $3$.
In particular,
\[
\tag{\ref{prl2}.6}\label{prl2006}
\exp(2^{-1}\pi \tilde{A}_1^1(1)) h^1(\xi_1,\xi_2,\xi_3;x_1,x_2,x_3)
=h^1(\xi_1,\xi_3,\xi_2;-\overline{x_1},-\overline{x_3},\overline{x_2})\]
with $\xi_i \in\mathbb{R}$ and $x_i\in{\bf O}$.
Using (\ref{prl2005}),
we have the following lemma.

\begin{lemma}\label{prl2-06}
Let $t \in \mathbb{R}$.
\begin{align*}
&a_t(r(-E_1+E_2)+sP^-+u(E-E_3))\\
=&re^{-2t}(-E_1+E_2)+(r\sinh 2t+se^{2t})P^-
+u(E-E_3)
\end{align*}
where $r,s,u\in\mathbb{R}$.
\end{lemma}
\medskip

\begin{lemma}\label{prl2-07}
For any
$m \in M$,
$t \in \mathbb{R}$, and $n \in N^+$,
\[
m a_t n P^-=e^{2 t}P^-.
\]
Furthermore, $A \cap MN^+=\{1\}$
and $M \cap AN^+=\{1\}$.
\end{lemma}
\begin{proof}
From (\ref{itd019}) and Lemma~\ref{prl2-06},
$m a_t n P^-=e^{2 t}P^-$.
Suppose $a_t=m n$ 
for some $t \in \mathbb{R}$, $m \in M$, and $n \in N^+$.
From the above equation and (\ref{itd019}),
$e^{2 t}P^-=a_t P^- =mn P^-= P^-$.
Thus $t=0$, and $A \cap MN^+=\{1\}$.
Similarly,
suppose $m=a_t n$ 
for some $m \in M$, $t \in \mathbb{R}$, and $n \in N^+$.
Then $ P^- =m P^-= a_t n P^-=e^{2 t}P^-$.
Thus $t=0$, and $M \cap AN^+=\{1\}$.
\end{proof}
\medskip

\begin{lemma}\label{prl2-08}
$(\mathrm{F}_{4(-20)})_{\sigma P^-}=MN^-$.
\end{lemma}
\begin{proof}
Because of
$M \subset K=(\mathrm{F}_{4(-20)})^{\tilde{\sigma}}$,
$\sigma M=M \sigma$.
Using (\ref{itd019}), 
$(\mathrm{F}_{4(-20)})_{\sigma P^-}=\sigma
(\mathrm{F}_{4(-20)})_{P^-} \sigma^{-1} 
=\sigma MN^+\sigma^{-1}
=M\tilde{\sigma}(N^+)=
MN^-$.
\end{proof}
\medskip

\begin{lemma}\label{prl2-09}
{\rm (1)}
For any $t \in \mathbb{R}$, $x \in {\bf O}$,
and $p \in \mathrm{Im}{\bf O}$,
\[
a_t (\mathcal{G}_1(x)+\mathcal{G}_2(p))
a_t^{-1}
=\mathcal{G}_1(e^{t} x) + \mathcal{G}_2(e^{2t} p).\]
{\rm (2)}
$A N^+ = N^+ A$.
Furthermore,
$AN^+$ is a subgroup of $\mathrm{F}_{4(-20)}$.
\smallskip

\noindent
{\rm (3)}
$M A N^+$ is a subgroup of $\mathrm{F}_{4(-20)}$.
\end{lemma}
\begin{proof}
(1)
Set 
$T(t) \in
 \mathrm{GL}_{\mathbb{R}}(\mathfrak{f}_{4(-20)})$
and 
$\mathrm{ad}_{\tilde{A}_3^1(1)} \in
\mathrm{End}_{\mathbb{R}}(\mathfrak{f}_{4(-20)})$
as
$T(t) \phi:=a_t~\phi~a_t^{-1}$ 
and
$\mathrm{ad}_{\tilde{A}_3^1(1)}\phi:=[\tilde{A}_3^1(1),\phi]$
for $\phi \in 
\mathfrak{f}_{4(-20)}$,
respectively.
Then  
$T(t)=\exp (t~\mathrm{ad}_{\tilde{A}_3^1(1)})$,
and using (\ref{itd013}) and (\ref{itd014}),
$T(t) \mathcal{G}_1(x)
=(\sum (t~\mathrm{ad}_{\tilde{A}_3^1(1)})^n/n!)
\mathcal{G}_1(x)
= \mathcal{G}_1\left( (\sum (1/n!)t^n)x \right)
=\mathcal{G}_1(e^{t} x)$
and 
$T(t) \mathcal{G}_2(p)
=(\sum (t~\mathrm{ad}_{\tilde{A}_3^1(1)})^n/n!) 
\mathcal{G}_2(p)
= \mathcal{G}_1\left( (\sum (1/n!)(2 t)^n) p \right)
=\mathcal{G}_2(e^{2 t} p)$.
Thus we obtain (1).

(2) 
From (1),
$a_t n
=a_t n a_t^{-1} a_t
=\exp \left(a_t(\mathcal{G}_1(x)+\mathcal{G}_2(p))a_t^{-1}
\right)a_t=\exp (\mathcal{G}_1(e^t x)
+\mathcal{G}_2(e^{2 t} p))a_t $
and  $n a_t
=a_t a_t^{-1} n a_t
=a_t \exp (\mathcal{G}_1(e^{-t} x)
+\mathcal{G}_2(e^{-2 t} p))$.
This implies that $AN^+=NA^+$.
Therefore $(a_t n)^{-1}(a_s n') \in AN^+$ for all $s,t \in \mathbb{R}$
and $n, n' \in N^+$,
so that $AN^+$ is a subgroup.

(3) Because of
(\ref{itd010}), 
$M N^+ =N^+ M$,
and 
$A N^+ =N^+ A$,
we get $(m a_t n)^{-1}(m' a_s n') \in M A N^+$
for all $m, m' \in M$, $s, t \in \mathbb{R}$, and $n,n' \in N^+$.
Thus $MAN^+$ is a subgroup of $\mathrm{F}_{4(-20)}$.
\end{proof}
\medskip

\begin{lemma}\label{prl2-10}
Let $k \in K$, $k_{\epsilon} \in K_{\epsilon}$,
$m \in M$,
$t \in \mathbb{R}$,
$n \in N^+$, and $z \in N^-$.
\begin{align*}
\tag{1}
(k a_t n P^-|E_1)&=-e^{2t}.\\
\tag{2}
(k_{\epsilon} a_t n P^-|E_2)&=e^{2 t}.\\
\tag{3}
(z m a_t n P^-|\sigma P^-)&=4 e^{2 t}.
\end{align*}
\end{lemma}
\begin{proof} From (\ref{itd008}), (\ref{itd020}),
Lemmas~\ref{prl2-07} and \ref{prl2-08},
it follows that
\begin{gather*}
(k a_t n P^-|E_1)=(a_t n P^-|k^{-1}E_1)
=e^{2 t}(P^-|E_1)=-e^{2 t},\\
(k_{\epsilon} a_t n P^-|E_2)
=(a_t n P^-|k_{\epsilon}^{-1}E_2)
=e^{2 t}(P^-|E_2)=e^{2 t},\\
(z m a_t n P^-|\sigma P^-)
=(m a_t n P^-|z^{-1} \sigma P^-)
=e^{2 t}(P^-|\sigma P^-)=4 e^{2 t}.
\end{gather*}
\end{proof}
\medskip

\begin{lemma}\label{prl2-11}
$M =
(\mathrm{F}_{4(-20)})_{P^-,E_j}
=(\mathrm{F}_{4(-20)})_{P^-,\sigma P^-}$
with $j \in \{1,2\}$.
\end{lemma}
\begin{proof} 
If $j=1$ then $k=2$, and if $j=2$ then $k=1$.
Note $P^-=-E_1+E_2+F_3^1(1)$,
$\sigma P^-=-E_1+E_2+F_3^1(-1)$,
and
$M=(\mathrm{F}_{4(-20)})_{E_1,E_2,E_3,F_3^1(1)}$.
Obviously,
$M \subset (\mathrm{F}_{4(-20)})_{P^-,E_j}$.
Conversely, fix
$g \in (\mathrm{F}_{4(-20)})_{E_j,P^-}$.
Now
$((-1)^{j+1}E_j + P^-)^{\times 2} 
=E_3$.
Then 
$g E_3 = g((-1)^{j+1}E_j + P^-)^{\times 2}
=(g((-1)^{j+1}E_j + P^-))^{\times 2}=E_3$,
and $g E_k=g (E - E_j-E_3)
=E - E_j-E_3
=E_k$.
Therefore $g E_i=E_i$ for all $i \in \{1,2,3\}$,
and $g F_3^1(1)=g(P^+ + E_1-E_2)
=P^- + E_1-E_2=F_3^1(1)$.
Then $g \in M$,
so that $(\mathrm{F}_{4(-20)})_{P^-,E_j}
\subset M$.
Thus $M=(\mathrm{F}_{4(-20)})_{P^-,E_j}$.

Obviously
$M \subset
(\mathrm{F}_{4(-20)})_{P^-,\sigma P^-}$.
Conversely,
fix $g \in (\mathrm{F}_{4(-20)})_{P^-,\sigma P^-}$.
Because of $-E_1+E_2=2^{-1}(P^- - \sigma P^-)$,
$(-E_1+E_2)^{\times 2}=-E_3$,
$F_3^1(1)=P^- - (-E_1+E_2)$, 
$E_1=2^{-1}(E-(-E_1+E_2)-E_3)$,
and $E_2=2^{-1}(E+(-E_1+E_2)-E_3)$,
we sequentially get $g (-E_1+E_2) = -E_1+E_2$,
$g E_3
=E_3$, $g F_3^1(1) = F_3^1(1)$, $g E_1=E_1$, and $g E_2=E_2$.
Thus $g \in M$, and so 
$(\mathrm{F}_{4(-20)})_{P^-,\sigma P^-}
\subset M$. Hence 
$M=(\mathrm{F}_{4(-20)})_{P^-,\sigma P^-}$.
\end{proof}
\medskip

\begin{lemma}\label{prl2-12}
Let $K'=K$ or $K_{\epsilon}$.
\begin{gather*}
\tag{1}
\mathrm{D}_4
\cap N^{\pm} = \{ 1 \}\quad (resp),\\
\tag{2}
K' \cap AN^+ = \{ 1 \},\\
\tag{3}
N^- \cap M A N^+ = \{ 1 \}.
\end{gather*}
\end{lemma}
\begin{proof}
(1)
Fix $n \in \mathrm{D}_4 \cap N^+$.
Then $n \in \mathrm{D}_4
\subset (\mathrm{F}_{4(-20)})_{E_3,-E_1+E_2}$.
Now, 
$n=\exp  \mathcal{G}_1(x)  \exp \mathcal{G}_2(p) $
for some $x \in {\bf O}$ and $p \in \mathrm{Im} {\bf O}$.
Using (\ref{prl2002}) and (\ref{prl2003}),
$E_3=n E_3 = \exp  \mathcal{G}_1(x)  E_3 
=E_3 + Q^+(x) + (x|x) P^-$. 
Then  $x=0$ by (\ref{prl2001}).
Therefore from (\ref{prl2002}), $-E_1+E_2=n (-E_1+E_2)
=\exp \mathcal{G}_2(p)  (-E_1 + E_2)= (-E_1 + E_2)
+ F_3^1(-2 p) + (p|p)P^-$.
Then $p=0$.
Thus $n = 1$, and 
$\mathrm{D}_4
\cap N^+ = \{ 1 \}$. 
Because of $\mathrm{D}_4 \subset K=
(\mathrm{F}_{4(-20)})^{\tilde{\sigma}}$,
$\tilde{\sigma}(\mathrm{D}_4)
=\mathrm{D}_4$.
Then
from $\tilde{\sigma}(N^+)=N^-$,
$\mathrm{D}_4
 \cap N^-=
\tilde{\sigma}(\mathrm{D}_4
 \cap N^+)=\{ 1 \}$.

(2)
Take $j=1$ if $K'=K$,
and $j=2$ if $K'=K_{\epsilon}$.
Suppose $k'=a_t n$
for some $k \in K'$, $t \in \mathbb{R}$, and $n \in N^+$.
Using Lemma~\ref{prl2-10}(1)(2),
(\ref{itd008}), and (\ref{itd020}),
$(-1)^j e^{2t}=(a_t n P^-|E_j)
=(P^-|k'^{-1}E_j)=(P^-|E_j)
=(-1)^j$. 
Therefore $t=0$, and $K' \cap A N^+ \subset K' \cap N^+$.
Next, using
(\ref{itd008}), (\ref{itd020}), and (\ref{itd019}),
$K' \cap N^+ \subset
(\mathrm{F}_{4(-20)})_{E_j,P^-}$,
and
from Lemma~\ref{prl2-11}, 
$K' \cap N^+ \subset
(\mathrm{F}_{4(-20)})_{E_j,P^-}
\cap N^+=
M \cap N^+$.
Therefore because of  $M
\subset \mathrm{D}_4$
and (1),
$\{1\}\subset K' \cap A N^+ \subset K' \cap N^+
\subset M \cap N^+
\subset \mathrm{D}_4 \cap N^+=\{1\}$.
Hence $K' \cap AN^+=\{1\}$.

(3)
Suppose $z=m a_t n$
for some $z \in N^-$, $m \in M$, $t \in \mathbb{R}$, and $n \in N^+$.
Using Lemmas~\ref{prl2-10}(3)
and \ref{prl2-08},
$4 e^{2t}=(m a_t n P^-|\sigma P^-)
=(P^-|z^{-1}\sigma P^-)=(P^-|\sigma P^-)
=4$. 
Therefore $t=0$, 
so that
 $N^- \cap M A N^+ \subset N^- \cap M N^+$.
Next, 
using (\ref{itd019}) and Lemma~\ref{prl2-08},
$N^- \cap M N^+ \subset
(\mathrm{F}_{4(-20)})_{P^-,\sigma P^-}$,
and from Lemma~\ref{prl2-11},
$N^- \cap M N^+ \subset
N^- \cap
(\mathrm{F}_{4(-20)})_{P^-,\sigma P^-}
= N^- \cap M$.
Therefore because of $M
\subset \mathrm{D}_4$ and (1),
$\{1 \}\subset N^- \cap M A N^+ \subset N^- \cap M N^+
\subset M 
\cap N^- \subset
\mathrm{D}_4 \cap N^-
=\{1\}$.
Hence $N^- \cap M A N^+=\{1\}$.
\end{proof}
\medskip

\begin{lemma}\label{prl2-13}
{\rm (1)} If $k a_t n = k' a_s n'$ with
$k, k' \in K$, $t,s \in \mathbb{R}$, and $n,n' \in N^+$
then $k=k'$, $t=s$, and $n=n'$.
\smallskip

\noindent 
{\rm (2)} If $k_{\epsilon} a_t n = k_{\epsilon}' a_s n'$ 
with
$k_{\epsilon}, k_{\epsilon}' \in K_{\epsilon}$, 
$t,s \in \mathbb{R}$, and $n,n' \in N^+$
then $k_{\epsilon}=k_{\epsilon}'$, $t=s$, and $n=n'$.
\smallskip

\noindent
{\rm (3)}
If $z m a_t n = z' m' a_s n'$ with
$z, z' \in N^-$, $m,m' \in M$,
$t,s \in \mathbb{R}$, and $n,n' \in N^+$
then $z=z'$, $m=m'$, $t=s$, and $n=n'$.
\end{lemma}
\begin{proof} (1) 
From Lemma~\ref{prl2-09}(2),
$(a_s n')(a_t n)^{-1} \in A N^+$,
so that 
$k'^{-1}k = (a_s n')(a_t n)^{-1} \in K \cap A N^+$.
Using Lemma~\ref{prl2-12}(2),
$k=k'$
and
$a_t n = a_s n'$.
Next, because of
$a_s^{-1} a_t = n n'^{-1} \in A \cap N^+$
and 
Lemma~\ref{prl2-07},
$a_t = a_s \Leftrightarrow t=s$
and
$n=n'$.
Hence we obtain (1).
Similarly,
substituting $K$ for $K_{\epsilon}$,
we obtain (2).

(3) By Lemma~\ref{prl2-09}(3),
$(m' a_s n')(m a_t n)^{-1} \in M A N^+$,
so that 
$z'^{-1}z = (m' a_s n')(m a_t n)^{-1} \in N^- \cap M A N^+$.
Using Lemma~\ref{prl2-12}(3),
$z=z'$
and
$m a_t n =m'  a_s n'$.
Next, by Lemma~\ref{prl2-09}(2),
$(a_s n')(a_t n)^{-1} \in A N^+$,
so that 
$m'^{-1}m = (a_s n')(a_t n)^{-1} \in M \cap A N^+$.
Using
Lemma~\ref{prl2-07},
$m=m'$
and
$a_t n =a_s n'$.
Last, because of
$a_s^{-1} a_t = n n'^{-1} \in A \cap N^+$
and 
Lemma~\ref{prl2-07},
$a_t = a_s \Leftrightarrow t=s$
and
$n=n'$.
Hence  we obtain (3).
\end{proof}
\medskip

\begin{lemma}\label{prl2-14}
{\rm (1)} For any $X \in \mathcal{H}$
and $Y \in \mathcal{N}_1^-$,
$(X|Y) < 0$.
\smallskip

\noindent
{\rm (2)} For any $X \in \mathcal{H'}$
and $Y \in \mathcal{N}_1^-$,
$(X|Y) \geq 0$.
\smallskip

\noindent
{\rm (3)} 
For any $X,Y \in \mathcal{N}_1^-$,
$(X|Y) \geq 0$.
Moreover,
$(X|Y)=0$ if and only if $X=s Y$ for some $s>0$.
\end{lemma}
\begin{proof} 
(1)  Using (\ref{itd004}),
$X= g E_1$ for some
$g \in \mathrm{F}_{4(-20)}$.
Then 
from (\ref{itd007}),
$g^{-1} Y\in \mathcal{N}_1^-$,
and from the definition of
$\mathcal{N}_1^-$, 
we obtain that
$(X|Y)=(g E_1|Y)
=(E_1|g^{-1} Y)<0$.

(2) Suppose that $c=(X|Y)<0$.
Using (\ref{itd005}),
$X= g E_2$ for some
$g \in \mathrm{F}_{4(-20)}$.
Put $Z=g^{-1} Y$.
From (\ref{itd007}),
$Z \in \mathcal{N}_1^-$.
Now, because of $c=(g E_2|Y)=(E_2|Z)$,
$Z=h^1(\xi_1,c,\xi_3;x_1,x_2,x_3)$
for some $\xi_i\in \mathbb{R}$ and $x_i\in {\bf O}$.
Because of $Z \in \mathcal{N}_1^-$,
$\xi_1=(E_1|Z)< 0$ and
$\xi_1c+(x_3|x_3)=(Z^{\times 2})_{E_3}=0$.
Then 
$0=\xi_1c+(x_3|x_3)>0$,
and it is a contradiction.
Thus $c \geq 0$, and so (2) follows.

(3) 
Suppose that $(X|Y)<0$.
Using (\ref{itd007}),
$Y= g P^-$ for some
$g \in \mathrm{F}_{4(-20)}$.
Put $Z=g^{-1} X$.
From (\ref{itd007}),
$Z \in \mathcal{N}_1^-$.
Set 
$Z=\sum_{i=1}^3(\eta_i E_i + F_i^1(y_i))$
with $\eta_i \in \mathbb{R}$ and $y_i \in {\bf O}$,
and put $r=(y_3|1)$.
Then $-\eta_1+\eta_2-2r=(Z|P^-)=(X|Y)<0$.
Because of $Z \in \mathcal{N}_1^-$,
$\eta_1=(E_1|Z)<0$ and
$\eta_1\eta_2+(y_3|y_3)=(Z^{\times 2})_{E_3}=0$.
Then $\eta_1\eta_2=-(y_3|y_3)\leq 0$.
Therefore from $\eta_1<0$, 
$\eta_2 \geq 0$,
so that
$2r>\eta_2-\eta_1>0$.
Now, using Schwarz inequality,
$r^2 =(y_3|1)^2 \leq (y_3|y_3)(1|1)=(y_3|y_3)$.
Therefore because of
$\eta_1\eta_2+(y_3|y_3)=0$,
$4r^2>(\eta_2-\eta_1)^2=(\eta_2-\eta_1)^2
+4\bigl(\eta_1\eta_2+(y_3|y_3)\bigr)
=(\eta_2+\eta_1)^2+4(y_3|y_3)
\geq (\eta_2+\eta_1)^2+4r^2 \geq 4r^2$.
It is a contradiction, and so $(X|Y) \geq 0$.

If $X=s Y$ then $(X|Y)=0$. 
Conversely, suppose that $(X|Y)=0$.
Using (\ref{itd007}),
$Y= g P^-$ for some
$g \in \mathrm{F}_{4(-20)}$.
Put $Z=g^{-1} X$.
From (\ref{itd007}), 
$Z \in \mathcal{N}_1^-$.
Because of $(Z|P^-)=(X|Y)=0$
and Lemma~\ref{prl2-01}(1),
$\{Z\}_{-E_1+E_2}=0$.
Then
by (\ref{prl2001}), 
$Z=sP^-+uE+vE_3
+F_3^1(p)+Q^+(x)+Q^-(y)$
for some $u,v\in\mathbb{R}$,
$p\in{\rm Im}{\bf O}$,
and $x,y \in {\bf O}$.
Setting $z=x+y$ and
$w=\overline{x}-\overline{y}$,
$Z=sP^-+uE+vE_3
+F_3^1(p)+F_1^1(z)+F_2^1(w)$.
Now, because of $Z \in\mathcal{N}_1^-$, 
$u^2+(p|p)=(Z^{\times 2})_{E_3}
=0$ and 
$3 u+v=\mathrm{tr}(Z)
=0$.
Then $u=p=v=0$,   
and $Z=
sP^-+F_1^1(z)+F_2^1(w)$.
Again, because of $Z\in\mathcal{N}_1^-$,
$-s=(Z|E_1)<0$,
$-(z|z)=(Z^{\times 2})_{E_1}=0$, and
$(w|w)=(Z^{\times 2})_{E_2}=0$.
Thus $Z=s P^-$ with $s>0$.
Therefore, multiplying $g$ from left,
$X=s Y$.
Hence we obtain (3).
\end{proof}
\medskip

\begin{lemma}\label{prl2-15}
{\rm (1)}
$\mathcal{H}=\mathcal{H}^{P^-}_{<0}
=\mathcal{H}^{P^-}_{\ne 0}$.
\smallskip

\noindent
{\rm (2)}
$\mathcal{H'}=\mathcal{H'}^{P^-}_{>0}
\coprod \mathcal{H'}^{P^-}_{=0}$.
Especially,
$\mathcal{H'}^{P^-}_{>0}
= \mathcal{H'}^{P^-}_{\ne 0}$.
\smallskip

\noindent
{\rm (3)}
$\mathcal{N}_1^-=(\mathcal{N}_1^-)^{E_1}_{<0}
=(\mathcal{N}_1^-)^{E_1}_{\ne 0}$.

\noindent
{\rm (4)}
$\mathcal{N}_1^-=(\mathcal{N}_1^-)^{E_2}_{>0}
\coprod (\mathcal{N}_1^-)^{E_2}_{=0}$.
Especially,
$(\mathcal{N}_1^-)^{E_2}_{>0}
= (\mathcal{N}_1^-)^{E_2}_{\ne 0}$.
\smallskip

\noindent
{\rm (5)}
$\mathcal{N}_1^-=
(\mathcal{N}_1^-)^{\sigma P^-}_{>0}
\coprod (\mathcal{N}_1^-)^{\sigma P^-}_{=0}$.
Especially,
$(\mathcal{N}_1^-)^{\sigma P^-}_{>0}
= (\mathcal{N}_1^-)^{\sigma P^-}_{\ne 0}$.
Furthermore,
$(\mathcal{N}_1^-)^{\sigma P^-}_{=0}
= \{s (\sigma P^-)|~s>0\}$.
\end{lemma}
\begin{proof} 
(1)
Because of $P^- \in \mathcal{N}_1^-$
and Lemma~\ref{prl2-14}(1), 
$(X|P^-)<0$ for all $X \in \mathcal{H}$,
and so (1) follows.

(2) Because of $P^- \in \mathcal{N}_1^-$
and Lemma~\ref{prl2-14}(2),
$(X|P^-)\geq 0$ for all $X \in \mathcal{H'}$,
and so (2) follows.

(3)
Because of $E_1 \in \mathcal{H}$
and Lemma~\ref{prl2-14}(1),
$(X|E_1)<0$ for all $X \in \mathcal{N}_1^-$,
and so (3) follows.

(4)
Because of $E_2 \in \mathcal{H'}$
and Lemma~\ref{prl2-14}(2),
$(X|E_2)\geq 0$ for all $X \in \mathcal{N}_1^-$,
and so (4) follows.

(5)
Because of $\sigma P^- \in \mathcal{N}_1^-$
and Lemma~\ref{prl2-14}(3), we obtain that
$(X|\sigma P^-)\geq 0$ for all $X \in \mathcal{N}_1^-$,
and that 
$X \in \mathcal{N}_1^-$ and
$(X|\sigma P^-)=0$ if and only if
$X
=s (\sigma P^-)$ for some $s>0$.
Thus (5) follows.
\end{proof}
\medskip

\begin{lemma}\label{prl2-16}
For $X,Y \in \mathcal{J}^1$, let 
$\mathcal{D}_{X,Y}=\{g \in \mathrm{F}_{4(-20)}|~(gX|Y)=0\}$.
Assume that
there exists $g_0 \in \mathrm{F}_{4(-20)}$
such that $(g_0X|Y)\ne 0$.
Then $\mathcal{D}_{X,Y}$ has no
interior points in $\mathrm{F}_{4(-20)}$.
Furthermore, the complement set $(\mathcal{D}_{X,Y})^c$ of
$\mathcal{D}_{X,Y}$ is an open dense submanifold of 
$\mathrm{F}_{4(-20)}$.
\end{lemma}
\begin{proof}
Set the function
$f(g)=(gX|Y)$ for $g \in \mathrm{F}_{4(-20)}$.
Note that
$\mathrm{F}_{4(-20)}$ is a connected real analytic manifold,
and that
$f$ is a real analytic function.
Therefore,
if the set $f^{-1}(0)$ has some 
interior points then $f \equiv 0$ on $\mathrm{F}_{4(-20)}$.
Since $f(g_0) \ne 0$ for some
$g_0 \in \mathrm{F}_{4(-20)}$,
$f^{-1}(0)$ has no
interior points.
Therefore  $(\mathcal{D}_{X,Y})^c$ is dense,
and
since $\mathcal{D}_{X,Y}$ is a closed set,
$(\mathcal{D}_{X,Y})^c$ is an open set. 
\end{proof}
\medskip

\begin{lemma}\label{prl2-17}
The equations {\rm (\ref{itd003})} hold.
\end{lemma}
\begin{proof}
Put $S_0=\{X \in\mathcal{H}|~(X|E_1)=1\}$.
Obviously, $\{E_1\}\subset S_0$.
Fix $X \in S_0$.
Because of $\mathrm{tr}(X)=1$ and $(X|E_1)=1$,
we can write
$X=h^1(1,\xi,-\xi;x_1, x_2, x_3)$
for some $\xi \in \mathbb{R}$ and $x_i \in {\bf O}$.
Because of $X^{\times 2}=0$, 
$-\xi^2-(x_1|x_1)=(X^{\times 2})_{E_1}=0$,
so that $\xi=x_1=0$.
Then $(x_i|x_i)=(X^{\times 2})_{E_i}=0$
for $i \in \{2,3\}$,
so that $x_i=0$.
Thus  $X=E_1$, and so 
(\ref{itd003})(i) follows.

Put $S_1=\{X \in\mathcal{H}'|~(X|E_1)=0\}$,
and $S_2=\{h^1(0, 1/2-\xi, 1/2+\xi;x, 0, 0)
\in \mathcal{J}^1|~
\xi^2+(x|x)=1/4\}$.
Taking $x=0$ and $\xi=\pm 1/2$,
we see $\{E_2,E_3\}\subset S_2$.
From direct calculations, $S_2 \subset S_1$.
Conversely, fix $X \in S_1$.
Because of $\mathrm{tr}(X)=1$ and $(X|E_1)=0$,
we can write $X=h^1(0, 1/2+\xi, 1/2-\xi; x_1, x_2, x_3)$
for some $\xi \in \mathbb{R}$ and $x_i \in {\bf O}$.
Then $1/4-\xi^2-(x_1|x_1)=(X^{\times 2})_{E_1}=0$
and $(x_i|x_i)=(X^{\times 2})_{E_i}=0$ with $i \in \{2,3\}$.
Therefore $X=h^1(0, 1/2+\xi, 1/2-\xi;x_1, 0, 0)$
with $\xi^2+(x_1|x_1)=1/4$,
and $X \in S_2$.
Thus $S_1 \subset S_2$, and so $S_1=S_2$.
\end{proof}

\section{The Iwasawa decomposition of 
$\mathrm{F}_{4(-20)}$.}\label{id}
Because of
$\mathcal{H}\simeq \mathrm{F}_{4(-20)}/K$,
we consider $AN^+$-orbits on $\mathcal{H}$
to give the 
Iwasawa decomposition of $\mathrm{F}_{4(-20)}$.

\begin{lemma}\label{id-01}
For all $X\in \mathcal{H}$,
\[a_{2^{-1}\log(-(P^-|X))} n_X X=E_1.\]
\end{lemma}
\begin{proof}
Put $t=2^{-1}\log(-(P^-|X))$.
By Lemma~\ref{prl2-15}(1),
$\mathcal{H}=\mathcal{H}^{P^-}_{< 0}
=\mathcal{H}^{P^-}_{\ne 0}
\subset (\mathcal{R}_1)^{P^-}_{\ne 0}$.
Then
$(P^-|X)<0$,
and $\log(-(P^-|X))$ is well-defined.
Using
$\mathrm{tr}(X)=1$ and Lemma~\ref{prl2-05},
$n_X X=r(-E_1+E_2)+s P^-+2^{-1}(E-E_3)$
where
$r=2^{-1}(P^-|X)$ and
$s=4^{-1}((P^-|X)^{-1}-(P^-|X))$.
Because of  $re^{-2t}=-2^{-1}$,
$r \sinh 2t+se^{2t}=0$,
and 
Lemma~\ref{prl2-06},
we get
$a_t n_X X=-2^{-1}(-E_1+E_2)+2^{-1}(E_1+E_2)
=E_1$.
\end{proof}

\begin{proof}
[{\bf Proof of Main-Theorem~\ref{itd-04}}]
Using (\ref{itd004}), 
$g^{-1}E_1\in \mathcal{H}$.
Then using Lemma~\ref{id-01}
and $a_{2^{-1}\log(-(g P^-|E_1))}
=a_{2^{-1}\log(-(P^-|g^{-1}E_1))}$,
\[a_{2^{-1}\log(-(g P^-|E_1))}
n_{g^{-1}E_1}g^{-1}E_1=
E_1.\]
Put $k=a_{2^{-1}\log(-(g P^-|E_1))}
n_{g^{-1}E_1}g^{-1}$.
Then
$k\in
(\mathrm{F}_{4(-20)})_{E_1}=K$
by (\ref{itd008}), 
and
 \[g=k^{-1}a_{2^{-1}\log(-(g P^-|E_1))}n_{g^{-1}E_1}
\in KAN^+.\] 
Set
$H(g)=2^{-1}\log(-(g P^-|E_1))
\tilde{A}_3^1(1) \in
\mathfrak{a}$,
$n_I(g)=n_{g^{-1}E_1} \in N^+$, and
$k(g)=k^{-1} \in K$,
respectively.
Then $g=k(g) \exp(H(g)) n_I(g)$,
and 
it follows from Lemma~\ref{prl2-13}(1) that
$H(g)$, $n_I(g)$, and $k(g)$
are uniquely determined.
Because of 
$(P^-|g^{-1}E_1)=(g P^-|E_1)$,
$(Q^+(e_i)|g^{-1}E_1)=(g Q^+(e_i)|E_1)$, 
and $(F_3^1(e_i)|g^{-1}E_1)=(g F_3^1(e_i)|E_1)$,
we see
\begin{gather*}
n_{g^{-1}E_1}=\exp\bigr(\mathcal{G}_1(
2^{-1}(\sum{}_{i=0}^7 (g Q^+(e_i)|E_1)e_i
)
/(g P^-|E_1))\\
+\mathcal{G}_2(
-2^{-1}(\sum{}_{i=1}^7 (g F_3^1(e_i)|E_1)
e_i)/(g P^-|E_1))
\bigr).
\end{gather*}
Moreover,
$k(g)=g n_I(g)^{-1}\exp (-H(g))$.
Hence the result follows.
\end{proof}
\medskip

Set $\tilde{D}_4:=\{
(g_1,g_2,g_3) \in \mathrm{SO}(8)^3|
(g_1 x)(g_2 y)=\overline{g_3 (\overline{x y})}
~~\text{for}~x,y \in {\bf O} \}$.
From \cite[Lemma~3.2(1)]{Na2012},
the following map $\varphi_0
:\tilde{D}_4 \to \mathrm{D}_4$
is a group isomorphism;
$\varphi_0(g_1,g_2,g_3)(\sum_{i=1}^3(\xi_i E_i
+F_i^1(x_i)))
=\sum_{i=1}^3(\xi_i E_i
+F_i^1(g_i x_i))$.
From \cite[(4.5)]{Na2012}, 
for $j \in \{1,2,3\}$ and $X=\sum{}_{i=1}^3
\bigl(\xi_i E_i+F_i^1(x_i)\bigr)$,
there exists $g_0=\varphi_0(g_1,g_2,g_3) \in 
\mathrm{D}_4$
such that
\begin{align*}
\tag{\ref{id}.1}\label{id001}
g_0 X
&=\left(\sum{}_{i=1}^3\xi_iE_i\right)
+F^1_j(r_0)+\sum{}_{k=1}^2F^1_{j+k}(g_{j+k}x_{j+k})\\
&\quad
\text{with}~ r_0=\sqrt{(x_j|x_j)}\in \mathbb{R}
\end{align*}
where
the index $j+k$ is counted modulo $3$.
\medskip 

\begin{proof}
[{\bf Proof of Theorem~\ref{itd-05}}]
For all $m \in M$, $t \in \mathbb{R}$, and $n \in N^+$,
using Lemma~\ref{prl2-07},
$m a_t n [P^-]=[e^{2 t} P^-]=[P^-]$,
so that
$M A N^+ \subset (\mathrm{F}_{4(-20)})_{[P^-]}$.
Conversely, fix $g \in (\mathrm{F}_{4(-20)})_{[P^-]}$.
Then $g P^-=s P^-$ for some $s>0$.
Because of $\mathrm{F}_{4(-20)}=K A N^+$,
$g$ can be expressed by 
$g= k a_t n$
with $k \in K$,
$t \in \mathbb{R}$,
and $n \in N^+$.
From Lemma~\ref{prl2-07},
$s P^-=g P^-=k (a_t n  P^-)=e^{2 t} k P^-$.
Now,
using Lemma~\ref{prl2-10}(1),
$-s=(s P^-|E_1)=(g P^-|E_1)
=(k a_t n  P^-|E_1)=-e^{2 t}$, so that $s=e^{2 t}$.
Then $k P^- =P^-$,
and from (\ref{itd008}) and Lemma~\ref{prl2-11}, 
$k \in K \cap (\mathrm{F}_{4(-20)})_{P^-}
=(\mathrm{F}_{4(-20)})_{E_1,P^-}=M$.
Thus
$g= k a_t n \in M A N^+$,
and $(\mathrm{F}_{4(-20)})_{[P^-]}
\subset MAN^+$.
Hence $(\mathrm{F}_{4(-20)})_{[P^-]}=MAN^+$,
and 
it follows from (\ref{itd007}) 
and $\mathcal{F}=\mathcal{N}_1^-/\sim$
that
\[
\mathcal{F}=\mathrm{F}_{4(-20)}\cdot [P^-]\simeq
\mathrm{F}_{4(-20)}/(\mathrm{F}_{4(-20)})_{[P^-]}
=\mathrm{F}_{4(-20)}/MAN^+.
\]
Next, let us show 
that $K$ transitively acts on $\mathcal{F}$.
Obviously $K$ acts on $\mathcal{F}$.
Fix $[X] \in \mathcal{F}$ with
$X \in\mathcal{N}_1^-$.
Using \cite[Lemma~5.2(4)]{Na2012}, 
there exists $k_1 \in K$
such that 
$k_1 X=h^1(-\xi,\xi,0;0,0,x)$
where
$\xi>0$, $x \in {\bf O}$,
and $\xi^2-(x|x)=0$.
Using (\ref{id001}),
there exists $k_2 \in \mathrm{D}_4 \subset K$
such that
$k_2 k_1 X=h^1(-\xi,\xi,0;0,0,\xi)
=\xi P^-$.
Thus
$k_2k_1[X]
=[\xi P^-]=[P^-]$,
and so $\mathcal{F}
=K\cdot [P^-]$.
Last, from $(\mathrm{F}_{4(-20)})_{[P^-]}=MAN^+$
and Lemma~\ref{prl2-13}(1), 
$K_{[P^-]}=(\mathrm{F}_{4(-20)})_{[P^-]} \cap K=(MAN^+)
\cap K=M$.
Thus from $\mathcal{F}=K\cdot [P^-]$,
(\ref{itd008}), and (\ref{itd011}), it follows that
\[\mathcal{F} \simeq
K/K_{[P^-]}=K/M=
\mathrm{Spin}(9)/\mathrm{Spin}(7).\]
\end{proof}
\medskip

We define the quadratic space $({\bf O}^2,Q)$
by the normal linear space ${\bf O}^2={\bf O}\times {\bf O}$
and $Q(x,y):=(x|x)+(y|y)$ for $x,y \in {\bf O}$,
and $S^{15}:=\{(x,y) \in {\bf O}^2|~Q(x,y)=1\}$.

\begin{proposition}\label{id-02}
\[ \mathcal{F}\simeq S^{15}.\]
Furthermore, 
$K/M
\simeq S^{15}$.
\end{proposition}
\begin{proof}
Set the map $f: S^{15} \to \mathcal{F}$
as
\[f(x,y):=[h^1(-1,(y|y),(x|x);\overline{xy},x,y)]
\quad \text{for}~(x,y) \in S^{15}.\]
Put $X=h^1(-1,(y|y),(x|x);\overline{xy},x,y)$.
From direct calculations, we get $X \in \mathcal{N}_1^-$.
Therefore $f$ is well-defined.
On the other hand,  the map $g: \mathcal{F} \to S^{15}$ set
as
\[g([h^1(\xi_1,\xi_2,\xi_3;x_1,x_2,x_3)])
:=(-\xi_1^{-1}x_2,-\xi_1^{-1}x_3)
\]
for $h^1(\xi_1,\xi_2,\xi_3;x_1,x_2,x_3) \in \mathcal{N}_1^-$.
Put $X=h^1(\xi_1,\xi_2,\xi_3;x_1,x_2,x_3)$, 
$x=-\xi_1^{-1}x_2$, and $y=-\xi_1^{-1}x_3$, respectively.
Because of 
$\xi_3\xi_1+(x_2|x_2)=(X^{\times 2})_{E_2}=0$,
$\xi_1\xi_2+(x_3|x_3)=(X^{\times 2})_{E_3}=0$,
and
$\xi_1+\xi_2+\xi_3=\mathrm{tr}(X)=0$, 
we get
$Q(x,y)=\xi_1^{-2}((x_2|x_2)+(x_3|x_3))
=\xi_1^{-2}(-\xi_3\xi_1-\xi_1\xi_2)
=\xi_1^{-1}(-\xi_2-\xi_3)
=\xi_1^{-1}\xi_1=1$.
Therefore $g$ is well-defined.

Now, it follows that $g \circ f(x,y)=(x,y)$ for all $(x,y) \in S^{15}$,
so that $g\circ f=id$.
On the other hand,
fix $X=h^1(\xi_1,\xi_2,\xi_3;x_1,x_2,x_3) \in \mathcal{N}_1^-$.
From $-\overline{x_2x_3}-\xi_1x_1=(X^{\times 2})_{F_1}=0$,
$\xi_1\xi_2+(x_3|x_3)=(X^{\times 2})_{E_3}=0$,
and $\xi_3\xi_1+(x_2|x_2)=(X^{\times 2})_{E_2}=0$,
we get
$x_1=-\xi_1^{-1}(\overline{x_2x_3})$,
$\xi_2=-\xi_1^{-1}(x_3|x_3)$,
and
$\xi_3=-\xi_1^{-1}(x_2|x_2)$.
Then
because of $\mathcal{F}=\mathcal{N}_1^-/\sim$
and $\xi_1=(X|E_1)<0$,
we see
\begin{align*}
&f \circ g ([h^1(\xi_1,\xi_2,\xi_3;x_1,x_2,x_3)])\\
=&[h^1(-1,\xi_1^{-2}(x_3|x_3),\xi_1^{-2}(x_2|x_2);
\xi_1^{-2}(\overline{x_2 x_3}),-\xi_1^{-1}x_2,-\xi_1^{-1}x_3)]\\
=&[h^1(\xi_1,-\xi_1^{-1}(x_3|x_3),-\xi_1^{-1}(x_2|x_2);
-\xi_1^{-1}(\overline{x_2 x_3}),x_2,x_3)]\\
=&[h^1(\xi_1,\xi_2,\xi_3;x_1,x_2,x_3)].
\end{align*}
Therefore  $f \circ g=id$.
Hence $\mathcal{F}\simeq S^{15}$,
and from Theorem~\ref{itd-05}(4),
$K/M\simeq S^{15}$.
\end{proof}

\begin{remark}{\rm
I.~Yokota has
proved $\mathrm{Spin}(9)/\mathrm{Spin}(7) \simeq S^{15}$
(\cite[Example~5.6]{YiJ1973}, \cite{Yi1968})
by realizing
$\mathrm{Spin}(9)$ and $\mathrm{Spin}(7)$
as stabilizers
of finite points
in
the compact exceptional Lie group $\mathrm{F}_4
:=Aut_{\mathbb{R}}(\mathcal{J})$
where $\mathcal{J}$ is an exceptional Jordan algebra,
and showing that
$\mathrm{Spin}(9)$ transitively acts on $S^{15}$
embed in $\mathcal{J}$.
In Proposition~\ref{id-02},
we give the other proof 
by using $\mathcal{F}=\mathcal{N}_1^-/\sim$
where $\mathcal{N}_1^-$ is an exceptional null cone.}
\end{remark}

\section{The $K_{\epsilon}-$Iwasawa decomposition
of $\mathrm{F}_{4(-20)}$.}
\label{ke}
Because of
$\mathcal{H'} 
\simeq \mathrm{F}_{4(-20)}/K_{\epsilon}$,
we consider $AN^+$-orbits on $\mathcal{H'}$
to give the $K_{\epsilon}-$Iwasawa decomposition
of $\mathrm{F}_{4(-20)}$.

\begin{lemma}\label{ke-01}
Assume that $X \in \mathcal{H'}^{P^-}_{\ne 0}$.
Then
\[
a_{2^{-1} \log((P^-|X))} n_X X=E_2.
\]
\end{lemma}
\begin{proof}
Put $t=2^{-1}\log((P^-|X))$.
By Lemma~\ref{prl2-15}(2),
$\mathcal{H'}^{P^-}_{> 0}
=\mathcal{H'}^{P^-}_{\ne 0}
\subset (\mathcal{R}_1)^{P^-}_{\ne 0}$.
Then
$(P^-|X)>0$,
and $\log((P^-|X))$ is well-defined.
Using
$\mathrm{tr}(X)=1$ and Lemma~\ref{prl2-05},
$n_X X=r(-E_1+E_2)+s P^-+2^{-1}(E-E_3)$
where
$r=2^{-1}(P^-|X)$ and
$s=4^{-1}((P^-|X)^{-1}-(P^-|X))$.
Because of  $re^{-2t}=2^{-1}$,
$r \sinh 2t+se^{2t}=0$,
and 
Lemma~\ref{prl2-06},
we get
$a_t n_X X=2^{-1}(-E_1+E_2)+2^{-1}(E_1+E_2)
=E_2$.
\end{proof}
\medskip

\begin{proof}
[{\bf Proof of 
Main Theorem~\ref{itd-06}}]
Put $\mathcal{D}=\{g \in \mathrm{F}_{4(-20)}|
~(g P^-|E_2)> 0\}$.
From (\ref{itd007}) and Lemma~\ref{prl2-15}(4),
we see
$\mathcal{D}=\{g \in \mathrm{F}_{4(-20)}|
~g P^-\in (\mathcal{N}_1^-)^{E_2}_{>0}\}
=\{g \in \mathrm{F}_{4(-20)}|
~g P^-\in (\mathcal{N}_1^-)^{E_2}_{ \ne 0}\}
=\{g \in \mathrm{F}_{4(-20)}|
~(g P^-|E_2)\ne 0\}$.
Now, from Lemma~\ref{prl2-10}(2),
$K_{\epsilon} A N^+\subset \mathcal{D}$.
Conversely,
fix $g\in \mathcal{D}$.
From (\ref{itd005}),
$g^{-1}E_2 \in \mathcal{H'}$,
and
$(P^-|g^{-1}E_2)=(gP^-|E_2)>0$,
so that $g^{-1}E_2\in \mathcal{H'}^{P^-}_{> 0}$.
Using Lemma~\ref{ke-01}
 and $a_{2^{-1} \log((g P^-|E_2))}
=a_{2^{-1} \log((P^-|g^{-1} E_2))}$, 
\[a_{2^{-1} \log((g P^-|E_2))}
n_{g^{-1}E_2}
g^{-1}E_2=E_2.\]
Put $k'=
a_{2^{-1} \log((g P^-|E_2))}n_{g^{-1}E_2}
g^{-1}$.
Then
$k'\in (\mathrm{F}_{4(-20)})_{E_2}
=K_{\epsilon}$
by
(\ref{itd020}), 
and
\[\tag{*}
g=k'^{-1}a_{2^{-1} \log((g P^-|E_2))}n_{g^{-1}E_2}
\in K_{\epsilon} A N^+.\]
Thus $\mathcal{D} \subset K_{\epsilon}AN^+$,
and so $\mathcal{D} = K_{\epsilon}AN^+$.
Since the identity element $1 \in 
\mathrm{F}_{4(-20)}$ is in $\mathcal{D}$
and the complement set $\mathcal{D}^c$
is given by
$\mathcal{D}^c=\{g \in \mathrm{F}_{4(-20)}|
(g P^-|E_2)= 0\}$,
applying Lemma~\ref{prl2-16},
$\mathcal{D} = K_{\epsilon}AN^+$ is an open dense
submanifold of $\mathrm{F}_{4(-20)}$.

From (*),  we set
$H_{\epsilon}(g)
=2^{-1} \log((g P^-|E_2))\tilde{A}_3^1(1) \in
\mathfrak{a}$,
$n_{\epsilon}(g)=n_{g^{-1}E_2}
\in N^+$, and
$k_{\epsilon}(g)=k'^{-1} \in K_{\epsilon}$,
respectively.
Then we
get
$g=k_{\epsilon}(g) \exp(H_{\epsilon}(g)) n_{\epsilon}(g)$,
and 
it follows from Lemma~\ref{prl2-13}(2)
that  $H_{\epsilon}(g)$, $n_{\epsilon}(g)$,
and $k_{\epsilon}(g)$ are uniquely determined.
Since
$(X|g^{-1}Y)=(g X|Y)$ for all $X,Y \in \mathcal{J}^1$,
we see
\begin{gather*}
n_{g^{-1}E_2}
=\exp\bigl(\mathcal{G}_1(
2^{-1} (\sum{}_{i=0}^7 (g Q^+(e_i)|E_2)e_i
)
/(gP^-|E_2))\\
+\mathcal{G}_2(
-2^{-1} (\sum{}_{i=1}^7 (g F_3^1(e_i)|E_2)
e_i)/(gP^-|E_2))
\bigr).
\end{gather*}
Moreover, $k_{\epsilon}(g)
=gn_{\epsilon}(g)^{-1}\exp (-H_{\epsilon}(g))$.
Hence the result follows.
\end{proof}

\section{The Matsuki decomposition
of $\mathrm{F}_{4(-20)}$.}
\label{md}
For $X \in \mathcal{J}^1$,
we denote $L^{\times}(X) \in \mathrm{End}_{\mathbb{R}}(
\mathcal{J}^1)$ by $L^{\times}(X)Y=X\times Y$
for $Y \in \mathcal{J}^1$.
For $j \in \{1,2,3\}$ and $p,q \in \mathbb{R}$,
we denote the subspace
$(\mathcal{J}^1)^j_{p,q}$ of $\mathcal{J}^1$
by
\[(\mathcal{J}^1)^j_{p,q}:=\{
X \in \mathcal{J}^1|~
\sigma_j X=p X,\quad
L^{\times}(2 E_j) X =q X\}.\]

\begin{lemma}\label{md-01}
Let $j \in \{1,2,3\}$ and $p,q \in \mathbb{R}$.

\noindent
{\rm (1)}
For all $k \in (\mathrm{F}_{4(-20)})_{E_j}$,
\[L^{\times}(2E_j) \cdot k=k \cdot L^{\times}(2E_j).\]

\noindent
{\rm (2)} The stabilizer
$(\mathrm{F}_{4(-20)})_{E_j}$ invariants the space
$(\mathcal{J}^1)^j_{p,q}$.
\end{lemma}
\begin{proof}
(1) It follows from
$L^{\times}(2 E_j) ( k X)=2E_j \times (k X)
=k(2 E_j \times X)=k ( L^{\times}(2 E_j) X)$
for all $X \in \mathcal{J}^1$.

(2) From \cite[Proposition~4.14]{Na2012} and (1),
we see that
$k \cdot \sigma_j=\sigma_j \cdot k$
and $L^{\times}(2E_j) \cdot k=k \cdot L^{\times}(2E_j)$
for all $k \in (\mathrm{F}_{4(-20)})_{E_j}$.
Hence (2) follows.
\end{proof}
\medskip

By direct calculations, we have the following lemma.
\begin{lemma}\label{md-02}
Let $j \in \{1,2,3\}$.
\[
\mathcal{J}^1
=(\mathcal{J}^1)^j_{-1,0}\oplus
(\mathcal{J}^1)^j_{1,0}
\oplus
(\mathcal{J}^1)^j_{1,1}
\oplus
(\mathcal{J}^1)^j_{1,-1}
\]
where
\begin{align*}
(\mathcal{J}^1)^j_{-1,0}&=
\{F_{j+1}^1(x_{j+1})+F_{j+2}^1(x_{j+2})|~
x_{j+1},x_{j+2} \in {\bf O}\},\\
(\mathcal{J}^1)^j_{1,0}&=\{p E_j|~
p \in \mathbb{R}\},\quad
(\mathcal{J}^1)^j_{1,1}=\{q (E-E_j)|~
q \in \mathbb{R}\},\\
(\mathcal{J}^1)^j_{1,-1}&=\{\xi (E_{j+1}-E_{j+2}) + F_j^1(x_j)|~
\xi \in \mathbb{R},~x_j \in {\bf O}\}
\end{align*}
and indexes $j,~j+1,~j+2$ are counted modulo $3$.
\end{lemma}
\medskip

Let $j \in \{2,3\}$.
For $X \in \mathcal{J}^1$,
we denote the quadratic form $Q$ by
$Q(Y):=-\mathrm{tr}(Y^{\times 2})$
for $Y \in \mathcal{J}^1$,
and
$\mathcal{S}^{8,1}_j:=\{X \in (\mathcal{J}^1)^j_{1,-1}|
~Q(X)=1\}
=\{\xi (E_{j+1}-E_{j+2})+F_j^1(x)|~\xi\in\mathbb{R}, x \in 
{\bf O},~
\xi^2-(x|x)=1\}$.

\begin{lemma}\label{md-03}
Let $j \in \{2,3\}$
and indexes $j,~j+1,~j+2$ be counted modulo $3$.
$\mathcal{S}^{8,1}_j$ decomposes
into the following two $(\mathrm{F}_{4(-20)})_{E_j}$-orbits:
\[
\mathcal{S}^{8,1}_j=(\mathrm{F}_{4(-20)})_{E_j} \cdot (E_{j+1}-E_{j+2})
\coprod (\mathrm{F}_{4(-20)})_{E_j} \cdot (-E_{j+1}+E_{j+2}).\]
\end{lemma}
\begin{proof}
From \cite[Lemmas~4.2 and 4.6]{Na2012},
\[\mathcal{S}^{8,1}_3
=(\mathrm{F}_{4(-20)})_{E_3}  \cdot (E_1-E_2)
\coprod (\mathrm{F}_{4(-20)})_{E_3}  \cdot(-E_1+E_2).\]
Put $g_0=\exp(
2^{-1} \pi \tilde{A}_1^1(1))$.
Multiplying $g_0$ from the left, 
we have
\[\mathcal{S}^{8,1}_2=
(\mathrm{F}_{4(-20)})_{E_2}  \cdot (E_1-E_3)
\coprod \mathrm{F}_{4(-20)})_{E_2}  \cdot (-E_1+E_3).\]
Here, using $g_0 \sigma_3 g_0^{-1}=\sigma_2$,
we can prove $g_0 \mathcal{S}^{8,1}_3=\mathcal{S}^{8,1}_2$.
\end{proof}

\begin{lemma}\label{md-04}
Let $X\in\mathcal{N}_1^-$.

{\rm (1)} If $(X|E_2)\ne 0$,
then
there exists $k_{\epsilon}\in K_{\epsilon}$ such that
$k_{\epsilon}X=rP_{12}^-$ for some $r>0$.
\smallskip

{\rm (2)} If $(X|E_2)= 0$,
then
there exists $k_{\epsilon}\in K_{\epsilon}$ such that
$k_{\epsilon}X=rP_{13}^-$ for some $r>0$.
\end{lemma}
\begin{proof} (1) 
From Lemma~\ref{md-02},
$X$ can be expressed by 
$X
=(F_3^1(x_3)+F_1^1(x_1))+p E_2+q(E-E_2)
+(\xi(E_3-E_1)+F_2^1(x_2))$
where 
$F_3^1(x_3)+F_1^1(x_1) \in (\mathcal{J}^1)^2_{-1,0}$,
$p E_2 \in (\mathcal{J}^1)^2_{1,0}$,
$q(E-E_2) \in (\mathcal{J}^1)^2_{1,1}$,
$\xi(E_3-E_1)+F_2^1(x_2) \in (\mathcal{J}^1)^2_{1,-1}$,
and $p=(X|E_2)\ne 0$
with
$p, q, \xi \in \mathbb{R}$
and $x_i \in {\bf O}$.
Because of $X\in\mathcal{N}_1^-$,
we see
$p+2q=\mathrm{tr}(X)=0$ and
$q^2-\xi^2+(x_2|x_2)=(X^{\times 2})_{E_2}=0$. 
Then $\xi^2-(x_2|x_2)=4^{-1}p^2> 0$.
Setting $r=2^{-1}|p|$, we can write
$\xi(E_3-E_1)+F_2^1(x_2)=r W$
for some 
$W\in \mathcal{S}^{8,1}_2$.
From Lemma~\ref{md-03} and (\ref{itd020}),
there exists
$k_0\in K_{\epsilon}=(\mathrm{F}_{4(-20)})_{E_2}$ such that
$k_0W=\epsilon(E_3-E_1)$ with $\epsilon=\pm 1$.
Because of $K_{\epsilon}=(\mathrm{F}_{4(-20)})_{E_2}$,
we get $k_0 (p E_2)=p E_2$
and
$k_0(q(E-E_2)=q(E-E_2)$.
And because of
 $F_3^1(x_3)+F_1^1(x_1) \in (\mathcal{J}^1)^2_{-1,0}$ 
and Lemma~\ref{md-01}(2),
we get
$k_0( F_3^1(x_3)+F_1^1(x_1))=
F_3^1(y_3)+F_1^1(y_1)$
for some $y_i \in {\bf O}$.
Therefore
$k_0X=
h^1(\eta_1,p,\eta_3;y_1,0,y_3)$
where $\eta_1=q-\epsilon r$ and $\eta_3=q+\epsilon r$.
Put $X'=k_0X$.
Because of $X' \in\mathcal{N}_1^-$
by (\ref{itd007}),
\begin{align*}
\text{(i)}&~~\eta_1=(E_1|X')<0,&\text{(ii)}&~~
\eta_1+p+\eta_3=\mathrm{tr}(X')=0,\\
\text{(iii)}&~~\eta_3\eta_1=(X'^{\times 2})_{E_2}=0,
&\text{(iv)}&~~
p\eta_3-(y_1|y_1)=(X'^{\times 2})_{E_1}=0,\\
\text{(v)}&~~\eta_1p+(y_3|y_3)=(X'^{\times 2})_{E_3}
=0.&&
\end{align*}
Form (i), (ii), and (iii), 
we get
$\eta_3=0$,
$\eta_1=-p$, and $p>0$.
Next by (iv) and (v), we get $y_1=0$ and $p=\sqrt{(y_3|y_3)}$.
Consequently $X'=h^1(-p,p,0;0,0,y_3)$ with 
$p=\sqrt{(y_3|y_3)}$.
Using (\ref{id001}),
there exists $k_1\in D_4\subset K_{\epsilon}$ such that
$k_1k_0 X=k_1X'=h^1(-p,p,0;0,0,p)=p P_{12}^-$.

(2) Because of $\mathrm{tr}(X)=0$ and $(X|E_2)=0$,
$X=h^1(-r,0,r;x_1,x_2,x_3)$
for some $r \in \mathbb{R}$ and $x_i\in{\bf O}$. 
Because of $X\in\mathcal{N}_1^-$,
we get
$-r=(E_1|X)<0$,
$-(x_1|x_1)=(X^{\times 2})_{E_1}
=0$,
$-r^2+(x_2|x_2)=(X^{\times 2})_{E_2}
=0$, and
$(x_3|x_3)=(X^{\times 2})_{E_3}
=0$.
Then 
$X=h^1(-r,0,r;0,x_2,0)$
with $r=\sqrt{(x_2|x_2)}$.
Using (\ref{id001}),
there exists 
$k'\in D_4\subset K_{\epsilon}$ such that
$k'X=h^1(-r,0,r;0,r,0)=rP_{13}^-$.
\end{proof}
\medskip

\begin{proof}[{\bf Proof of 
Theorem~\ref{itd-07}}]
Set $\mathcal{O}=\{[X]\in\mathcal{F}|~(X|E_2)\ne 0\}$,
and $\mathcal{O}'=\{[X]\in\mathcal{F}|~(X|E_2)= 0\}$.
Then $\mathcal{F}=\mathcal{O} \coprod \mathcal{O}'$. 
Because of $\mathcal{F}=\mathcal{N}_1^-/\sim$
and Lemma~\ref{prl2-15}(4),
$\mathcal{O}=\{[X]\in\mathcal{F}|~(X|E_2)> 0\}$.
For any $k \in K_{\epsilon}$,
$(k X |E_2)
=(X| k^{-1}E_2)=(X|E_2)$
 by  (\ref{itd020}),
so that
$K_{\epsilon}$ acts on $\mathcal{O}$
and $\mathcal{O}'$, respectively.
When $[X] \in \mathcal{O}$,
by Lemma~\ref{md-04}(1), 
there exists $k \in K_{\epsilon}$
such that
$k [X]=[k X]=[P_{12}^-]$.
And when $[X] \in \mathcal{O}'$,
by Lemma~\ref{md-04}(2), 
there exists $k'\in K_{\epsilon}$
such that
$k'[X]=[k'X]=[P_{13}^-]$.
Hence the result follows.
\end{proof}
\medskip

\begin{proof}[{\bf Proof of 
Theorem~\ref{itd-08}}]
Put $g_0=\exp(-2^{-1}\pi\tilde{A}_1^1(1))$, and
$\mathcal{D}=\{g\in \mathrm{F}_{4(-20)}|~(gP_{12}^-|E_2)= 0\}$.
Then $g_0^{-1}=\exp(2^{-1}\pi\tilde{A}_1^1(1))$,
and from (\ref{prl2006}),
$g_0^{-1}P_{13}^-=P_{12}^-$ and
$g_0^{-1}E_2=E_3$.
Fix $g\in \mathcal{D}$.
By (\ref{itd007}), $g P_{12}^- \in \mathcal{N}_1^-$,
and applying Theorem~\ref{itd-07},
$[g P_{12}^-]\in K_{\epsilon}\cdot [P_{13}^-]$.
Therefore
$k [g P_{12}^-]=[P_{13}^-]$
for some $k \in K_{\epsilon}$.
Then 
$g_0^{-1}k g[P_{12}^-]
=[g_0^{-1}P_{13}^-
]=[P_{12}^-]$, so that
$g_0^{-1} k g\in (\mathrm{F}_{4(-20)})_{[P_{12}^-]}$.
Using Theorem~\ref{itd-05}(1),
$g_0^{-1}k g
=m a_ tn$
for some $m\in M,~t \in \mathbb{R}$, and
$n\in N^+$.
Thus
$g=k^{-1}g_0
man
\in K_{\epsilon}g_0 M A N^+$,
and so
$\mathcal{D}
\subset  K_{\epsilon}g_0 M A N^+$.
Conversely, take 
$g=k g_0a_tmn\in K_{\epsilon}g_0 M A N^+$
with $k \in K_{\epsilon}$, 
$t\in\mathbb{R}$, and $n\in N^+$.
Because of Lemma~\ref{prl2-07}, (\ref{itd020}),
and $g_0^{-1}E_2=E_3$, we see
$
(g P_{12}^-|E_2)
=(m a_t n P_{12}^-|g_0^{-1} k^{-1}E_2)
=e^{2 t}(P_{12}^-|E_3)
=0
$.
Thus $g \in\mathcal{D}$,
and so $K_{\epsilon}g_0 MAN^+\subset
\mathcal{D}$.
Hence
$K_{\epsilon}g_0 MAN^+=
\mathcal{D}$.
Last, from $M \subset K_{\epsilon}$ and
Main Theorem~\ref{itd-06},
$K_{\epsilon}MAN^+=K_{\epsilon}AN^+
=\{g \in \mathrm{F}_{4(-20)}|~(g P_{12}^-|E_2)\ne 0\}$.
Thus
$\mathrm{F}_{4(-20)}
=\{g \in \mathrm{F}_{4(-20)}|~(g P_{12}^-|E_2)\ne 0\}
\coprod \{g\in \mathrm{F}_{4(-20)}|~(g P_{12}^-|E_2)= 0\}
=K_{\epsilon}MAN^+\coprod 
K_{\epsilon}g_0 MAN^+$.
\end{proof}

\section{The Bruhat and Gauss 
decompositions of $\mathrm{F}_{4(-20)}$.}
\label{gd}
Because of
$\mathcal{F}
\simeq \mathrm{F}_{4(-20)}/MAN^+$,
we consider $N^-$-orbits on $\mathcal{F}$
to give the Bruhat and Gauss decompositions
of $\mathrm{F}_{4(-20)}$.
For any 
$X \in (\mathcal{N}_1^-)^{\sigma P^-}_{\ne 0}$, 
denote
$z_X:=\tilde{\sigma}(n_{\sigma X}) \in N^-$.

\begin{lemma}\label{gd-01}
Assume that
$X\in (\mathcal{N}_1^-)^{\sigma P^-}_{\ne 0}$.
Then 
\[
z_X X=4^{-1}
(X|\sigma P^-) P^-.
\]
\end{lemma}
\begin{proof}
Since $(\sigma X|P^-)=(X|\sigma P^-) \ne 0$
and $\mathrm{tr}(\sigma X)=\mathrm{tr}(X)=0$,
applying Lemma~\ref{prl2-05},
\[n_{\sigma X}(\sigma X) =
4^{-1}(\sigma X|P^-)(-E_1+E_2+F_3^1(-1))
=4^{-1}(X|\sigma P^-)\sigma P^-.
\]
Thus
$z_X X=(\sigma n_{\sigma X}\sigma) X
=4^{-1}(X|\sigma P^-) P^-$.
\end{proof}
\medskip

\begin{proof}[{\bf Proof of 
Theorem~\ref{itd-09}}]
Set $\mathcal{O}=\{[X]\in\mathcal{F}|(X|\sigma P^-) 
> 0\}$, 
and $\mathcal{O}'=\{[X]\in\mathcal{F}|(X|\sigma P^-) 
= 0\}$.
Using Lemma~\ref{prl2-15}(5),
$\mathcal{O}=
\{[X]\in\mathcal{F}|X \in 
(\mathcal{N}_1^-)^{\sigma P^-}_{>0}\}
=\{[X]\in\mathcal{F}|X \in 
(\mathcal{N}_1^-)^{\sigma P^-}_{\ne 0}\}
=\{[X]\in\mathcal{F}|(X|\sigma P^-)\ne 0\}$
and
$\mathcal{O}'=
\{[X]\in\mathcal{F}|X \in 
(\mathcal{N}_1^-)^{\sigma P^-}_{=0}\}
=\{[\sigma P^-]\}$.
Therefore $\mathcal{F}=\mathcal{O}\coprod \mathcal{O}'$.
For any $z \in N^-$ and $[X] \in \mathcal{O}$,
using Lemma~\ref{prl2-08},
$(zX|\sigma P^-)=(X|z^{-1}(\sigma P^-))
=(X|\sigma P^-)>0$.
Therefore $N^-$ acts on $\mathcal{O}$.
Fix $[X] \in \mathcal{O}$.
Taking $z_X\in N^-$,
from Lemma~\ref{gd-01},
$z_X[X]=[4^{-1}(X|\sigma P^-)P^-]=[P^-]$.
Therefore $\mathcal{O}=N^- \cdot [P^-]$.
Next,
using Lemma~\ref{prl2-08},
$N^- \cdot [\sigma P^-]=\{[\sigma P^-]\}
=\mathcal{O}'$.
Therefore 
$\mathcal{F}
=\mathcal{O}\coprod \mathcal{O}'
=N^- \cdot [P^-] \coprod \{[\sigma P^-]\}
=N^- \cdot [P^-] \coprod N^- \cdot [\sigma P^-]$.
\end{proof}

\begin{proof}[{\bf Proof of 
Main Theorem~\ref{itd-10}}]
Put $\mathcal{D}
=\{g \in \mathrm{F}_{4(-20)}|(g P^-|\sigma P^-)>0\}$.
From (\ref{itd007}), $g P^-\in \mathcal{N}_1^-$,
and
using Lemma~\ref{prl2-15}(5),
$\mathcal{D}=\{g \in \mathrm{F}_{4(-20)}|
g P^-\in (\mathcal{N}_1^-)^{\sigma P^-}_{>0}\}
=\{g \in \mathrm{F}_{4(-20)}|
g P^-\in (\mathcal{N}_1^-)^{\sigma P^-}_{ \ne 0}\}
=\{g \in \mathrm{F}_{4(-20)}|
(g P^-|\sigma P^-)\ne 0\}$
and the complement set $\mathcal{D}^c$
of $\mathcal{D}$ is given by
$\mathcal{D}^c=\{g \in \mathrm{F}_{4(-20)}|
(g P^-|\sigma P^-)=0\}
=\{g \in \mathrm{F}_{4(-20)}|~gP^-
\in (\mathcal{N}_1^-)^{\sigma P^-}_{=0}\}
=\{g \in \mathrm{F}_{4(-20)}|
g[P^-]=[\sigma P^-]\}$.
First, let us show $\mathcal{D}=N^-MAN^+$.
From Lemma~\ref{prl2-10}(3),
$N^-MAN^+ \subset \mathcal{D}$.
Conversely,
fix $g \in \mathcal{D}$.
Then
$g P^- \in (\mathcal{N}_1^-)^{\sigma P^-}_{> 0}$,
and from Lemma~\ref{gd-01},
$(z_{g P^-}) g P^-=4^{-1}
(g P^-|\sigma P^-)P^-$ and $(g P^-|\sigma P^-)>0$.
Therefore $(z_{g P^-})g [ P^-]=[4^{-1}
(g P^-|\sigma P^-)P^-]=[P^-]$.
Using Theorem~\ref{itd-05}(1),
$(z_{g P^-})g=ma_tn$
for some $t \in\mathbb{R}$, $m \in M$, and $n \in N^+$.
Thus 
\[
\tag{*}
g=(z_{g P^-})^{-1}ma_tn
\in
N^-MAN^+,\]
and so $\mathcal{D} \subset N^-MAN^+$.
Hence $\mathcal{D}=N^-MAN^+$.
Since the identity element $1 \in 
\mathrm{F}_{4(-20)}$ is in $\mathcal{D}$, 
applying Lemma~\ref{prl2-16},
$\mathcal{D} = N^-MAN^+$ is an open dense
submanifold of $\mathrm{F}_{4(-20)}$.

Second, let us show $\mathcal{D}^c=\sigma MAN$.
Fix $\sigma ma_tn \in \sigma MAN$
with $m \in M$, $t\in \mathbb{R}$, and $n \in N^+$.
By Lemma~\ref{prl2-07},
$\sigma ma_tn P^-=e^{2t}(\sigma P^-)$,
so that $\sigma ma_tn [P^-]=[\sigma P^-]$.
Thus $\sigma MAN \subset \mathcal{D}^c$.
Conversely, fix $g \in \mathcal{D}^c$.
Because of $g [P^-]=[\sigma P^-]$,
$\sigma g [P^-]=[P^-]$.
Using Theorem~\ref{itd-05}(1),
$\sigma g \in MAN^+$.
Thus $g \in \sigma MAN^+$,
and so $\mathcal{D}^c\subset \sigma MAN^+$.
Hence $\mathcal{D}^c= \sigma MAN^+$,
and 
$\mathrm{F}_{4(-20)}
=\mathcal{D}
\coprod
\mathcal{D}^c
=N^-MAN^+\coprod \sigma MAN^+$.
Now, from $N^-=\tilde{\sigma}(N^+)=\sigma N^+ \sigma$
 and Lemma~\ref{prl2-09}(3),
it follows that
$N^-\sigma MAN^+
=\sigma N^+ MAN^+
=\sigma MAN^+$.

Next, from (*), 
set
$n^-_G(g)=(z_{g P^-})^{-1}$,
$a_G(g)=a_t$,
$n^+_G(g)=n$, and
$m_G(g)=m$,
respectively.
Then $g=n^-_G(g)m_G(g)a_G(g)n^+_G(g)$,
and 
it follows from
Lemma~\ref{prl2-13}(3)
that $a_G(g)$, $n^-_G(g)$, $n^+_G(g)$,
and  $m_G(g)$
are
uniquely determined.
Now, since $(z_{g P^-})g=ma_tn$
and the uniqueness of factors
of the Iwasawa decomposition of $\mathrm{F}_{4(20)}$,
$a_G(g)$,
$n^+_G(g)$, and
$m_G(g)$
are
given by
$a_G(g)=\exp \left(H((z_{g P^-})g)\tilde{A}_3^1(1)
\right)$,
$n^+_G(g)=n_I((z_{g P^-})g)$,
and $m_G(g)=k((z_{g P^-})g)$,
respectively.
Then these equations imply that (i),
(ii), (iii), and (iv).
Indeed,
using Lemma~\ref{gd-01},
\[-((z_{g P^-})gP^-|E_1)
=-4^{-1}(g P^-|\sigma P^-)(P^-|E_1)
=4^{-1}(g P^-|\sigma P^-),\]
so that $t=2^{-1}\log(4^{-1}(g P^-|\sigma P^-))$.
Because of
$\sigma Q^+(e_i)=Q^-(e_i)$, $\sigma F_3^1(e_i)=-F_3^1(e_i)$,
and
(\ref{itd018}), we see
\begin{align*}
(z_{g P^-})^{-1}
=&\tilde{\sigma}
\biggl(
\exp\left(\mathcal{G}_1
\bigl(-2^{-1}(\sum{}_{i=0}^7(Q^+(e_i)|\sigma g P^-)e_i)
/(P^-|\sigma g P^-)\bigr)\right.\\
&\left.+
\mathcal{G}_2
\bigl(2^{-1}(\sum{}_{i=1}^7
(F_3^1(e_i)|\sigma g P^-)e_i)/(P^-|\sigma g P^-)
\bigr)\right)\biggr)\\
=&
\exp\left( \mathcal{G}_{-1}
\bigl(-2^{-1}(\sum{}_{i=0}^7(Q^-(e_i)|g P^-)e_i)
/(g P^-|\sigma P^-)\bigr)\right.\\
&\left.+
\mathcal{G}_{-2}
\bigl(-2^{-1}(\sum{}_{i=1}^7
(F_3^1(e_i)|g P^-)e_i)/(g P^-|\sigma P^-)
\bigr)\right).
\end{align*}
Moreover, we get
$n^+_G(g)=n_I((z_{g P^-})g)=n_I(n^-_G(g)^{-1}g)$
and
$m_G(g)= 
(z_{g P^-})g n^{-1} a_t^{-1}=
n^-_G(g)^{-1} g n^+_G(g)^{-1} a_G(g)^{-1}$.
Hence the result follows.
\end{proof}

\appendix

\section{The 
explicit formula $c$-function
of $\mathrm{F}_{4(-20)}$.}

Recall  
$\mathfrak{a} 
=\{t \tilde{A}_3^1(1)|~t \in \mathbb{R}\}$.
Let $\mathfrak{a}^*$ be the dual of $\mathfrak{a}$,
and $\mathfrak{a}_\mathbb{C}^*$  the
complexification of $\mathfrak{a}^*$,
and recall $\alpha\in \Sigma \subset \mathfrak{a}^*
\subset\mathfrak{a}_\mathbb{C}^*$ satisfies
$\alpha(\tilde{A}_3^1(1))=1$.
Let
$B(\cdot,\cdot)$ be the Killing form of $\mathfrak{f}_{4(-20)}$.
For $\lambda\in\mathfrak{a}^*$,
we define the element $H_{\lambda}\in \mathfrak{a}$
by $B(H_{\lambda},H)=\lambda(H)$ for all
$H\in\mathfrak{a}$,
and the bilinear form $\langle\cdot,\cdot\rangle$
on $\mathfrak{a}_\mathbb{C}^*$ by setting 
$\langle\lambda_1,\lambda_2\rangle
:=B(H_{\lambda_1},H_{\lambda_2})$ 
and extending it to the whole of 
$\mathfrak{a}_\mathbb{C}^*$ by linearity.
For any $\lambda\in \mathfrak{a}_\mathbb{C}^*$,
we define $\lambda_{\alpha}\in \mathbb{C}$ by
\[\lambda_{\alpha}
:=(2\langle\lambda,\alpha\rangle)/
\langle\alpha,\alpha\rangle.\]
Because of ${\rm dim}~\mathfrak{a}_\mathbb{C}^*
={\rm dim}~\mathfrak{a}=1$,
$\lambda=2^{-1}\lambda_{\alpha}\alpha$.
We denote $m_{\alpha}:={\rm dim}~\mathfrak{g}_{\alpha}
={\rm dim}~{\bf O}=8$
and $m_{2\alpha}:={\rm dim}~\mathfrak{g}_{2\alpha}
={\rm dim}~({\rm Im}{\bf O})=7$,
and we define 
$\rho\in\mathfrak{a}_\mathbb{C}^*$ by
\[\rho
:=2^{-1}(({\rm dim}~\mathfrak{g}_{\alpha})\alpha
+({\rm dim}~\mathfrak{g}_{2\alpha})2\alpha)
=2^{-1}(m_{\alpha}
+2m_{2\alpha})\alpha.\]
For $\lambda\in\mathfrak{a}_\mathbb{C}^*$,
we consider the {\it spherical function}
 $\varphi_{\lambda}$ on $\mathrm{F}_{4(-20)}$
and the {\it $c$-function of Harish-Chandra} 
on $\mathfrak{a}_\mathbb{C}^*$.
From \cite{HC1958} (cf. \cite{OS1980}, 
\cite{Os2003}, \cite{Sj2003}),
$\varphi_{\lambda}$
is given by
\[\varphi_{\lambda}(g)
:=\int_{K}
e^{(\lambda-\rho)(H(g k))}dk=\int_{N^-}
e^{(\lambda-\rho)(H(g z))}
e^{-(\lambda+\rho)(H(z))}
dz\]
for $g\in \mathrm{F}_{4(-20)}$, and 
the function 
$c$ is given by
\[c(\lambda)
:=\int_{N^-}e^{-(\lambda+\rho)(H(z))}
dz.\]
Here the measure $dk$ on compact group $K$
is normalized such that the total measure is $1$, and
the Haar measures $dn$ of nilpotent group $N^+$
and $dz$ of nilpotent group $N^-$ 
satisfy that
\[\tilde{\sigma}(dn)=dz\quad\text{and}\quad
\int_{N^-}e^{-2\rho(H(z))}dz=1.\]

\begin{lemma}\label{gk-01}
Let $t \in \mathbb{R}$,
$p\in{\rm Im}{\bf O}$, 
$x\in{\bf O}$, and $t\in\mathbb{R}$.
Assume that
$z
=
\exp(\mathcal{G}_{-2}(p)
+\mathcal{G}_{-1}(x))\in N^-$.
Then
\begin{align*}
\tag{1}
H(a_tz)
&=2^{-1}\log(e^{-2t}((e^{2t}+(x|x))^2+4(p|p)))
\tilde{A}_3^1(1),\\
\tag{2}
H(z)&=2^{-1}\log((1+(x|x))^2+4(p|p))
\tilde{A}_3^1(1).
\end{align*}
\end{lemma}
\begin{proof} From (\ref{itd018}) and (\ref{itd015}),
 $z
=\sigma\exp\mathcal{G}_2(p)\exp\mathcal{G}_1(x)\sigma$.
Put $X=\sigma\exp\mathcal{G}_2(p)\exp\mathcal{G}_1(x)\sigma P^-$.
Using $\sigma P^-=2(-E_1+E_2)-P^-$, (\ref{prl2002}),
and 
(\ref{prl2003}), we calculate that
\begin{align*}
X
=&-(((x|x)+1)^2 +4(p|p))  E_1+(((x|x)-1)^2+4(p|p)) E_2\\
&+4(x|x)E_3+F_1^1(2((x|x) +2 p-1) x)\\
&+F_2^1(-2\overline{x}((x|x) -2 p+1) )+F_3^1(-(x|x)^2-4(p|p)+1+4 p).
\end{align*}
Set $X=h^1(\eta_1,\eta_2,\eta_3;y_1,y_2,y_3)$.
Using (\ref{prl2005}), 
we get
$(a_t z P^-|E_1)=(a_t X)_{E_1}
=2^{-1}((\eta_1+\eta_2)+(\eta_1-\eta_2)\cosh(2 t))
-(1|y_3)\sinh(2 t)$.
Because of $2^{-1}(\eta_1+\eta_2)=-2(x|x)$,
$2^{-1}(\eta_1-\eta_2)=-(x|x)^2-4(p|p)-1$,
and $(1|y_3)=-(x|x)^2-4(p|p)+1$,
we calculate that
\begin{align*}
(a_t z P^-|E_1)
&=-e^{-2t}
((x|x)^2+2e^{2t}(x|x)+e^{4t}+4(p|p))\\
&=-e^{-2t}
((e^{2t}+(x|x))^2+4(p|p)).
\end{align*}
Thus (1) follows from Main Theorem~\ref{itd-04}(i),
and substituting $t=0$ in (1),
we obtain (2).
\end{proof}
\medskip

\begin{proposition}\label{gk-02}
Assume $\lambda\in\mathfrak{a}_\mathbb{C}^*$,
\[a=4^{-1}(m_{\alpha}+2m_{2\alpha}+\lambda_{\alpha}),
\quad b=4^{-1}(m_{\alpha}+2m_{2\alpha}-\lambda_{\alpha}).\]
Then there exists the constant $C_0\in\mathbb{R}$ such that
\begin{align*}
\tag{1}
c(\lambda)&=C_0
\int_{\mathbb{R}^{m_{\alpha}}
\times \mathbb{R}^{m_{2\alpha}}}
((1+(x|x))^2+4(p|p))^{-
a}
dxdp,\\
\tag{2}
\varphi_{\lambda}(a_t)
&=C_0\int_{\mathbb{R}^{m_{\alpha}}
\times \mathbb{R}^{m_{2\alpha}}}e^{2bt}
((e^{2t}+(x|x))^2+4(p|p))^{-b}\cdot\\
&\quad((1+(x|x))^2+4(p|p))^{-a}
dxdp
\end{align*}
where the measure $dx$ and $dp$ are the Euclidean
measure.
\end{proposition}
\begin{proof}
It follows from Lemma~\ref{gk-01}.
\end{proof}
\medskip

From \cite[Lemma~7.2]{Na2012},
\[
B(\phi,\tilde{\sigma}\phi)
=-3\left(\sum{}_{i=1}^3\left((\sum{}_{j=0}^7
(D_ie_j|D_ie_j))+24(a_{i}|a_{i})\right)\right)\]
where
$\phi=d\varphi_0(D_1,D_2,D_3)
+\sum_{i=1}^3\tilde{A}_i^1(a_i)$
with $d\varphi_0(D_1,D_2,D_3)\in\mathfrak{d}_4$
and $a_i\in{\bf O}$.
We denote $Q(\phi):=-\langle \alpha,\alpha \rangle
B(\phi,\tilde{\sigma}\phi)$
for $\phi\in\mathfrak{f}_{4(-20)}$.
Then from 
direct calculations, we have
the following proposition.

\begin{proposition}\label{gk-03}
If $\lambda\in \mathfrak{a}^*_{\mathbb{C}}$,
$p\in{\rm Im}{\bf O}$, and $x\in{\bf O}$,
then $Q(\mathcal{G}_{-1}(x))=2(x|x)$ and 
$Q(\mathcal{G}_{-2}(p))=2(p|p).$
\end{proposition}
\medskip

\begin{corollary}\label{gk-04}
{\rm (\cite{Hl1970}, \cite{Sg1971}, 
cf. \cite[Lemma~4.12 and (4.27)]{OS1980})}.
Let $\lambda\in \mathfrak{a}^*_{\mathbb{C}}$.
For any $X\in \mathfrak{g}_{-\alpha}$ and
$Y\in \mathfrak{g}_{-2\alpha}$,
\[e^{\lambda(H(\exp(X+Y)))}
=\left((1+2^{-1}Q(X))^2
+2Q(Y)\right)^{4^{-1}\lambda_{\alpha}}.\]
\end{corollary}
\medskip

\begin{remark}\label{gk-05}{\rm 
(\cite{Hl1970}, \cite{Sg1971}, 
cf. \cite{Os2003}, \cite{Sj2003}).
From Proposition~\ref{gk-02}(1),
changing variables to polar coordinates, up to
the constant multiple, $c(\lambda)$ is equal to
\begin{align*}
&\int_0^{\infty}\int_0^{\infty}
t^{m_{\alpha}-1}s^{m_{2\alpha}-1}
((1+t^2)^2+s^2)^{-4^{-1}
(\lambda_{\alpha}+m_{\alpha}+2m_{2\alpha})}
dsdt\\
=&
\int_0^{\infty}\int_0^{\infty}
(s/(1+t^2))^{m_{2\alpha}-1}
(1+(s/(1+t^2))^2)^{-4^{-1}
(\lambda_{\alpha}+m_{\alpha}+2m_{2\alpha})}\\
&\quad
\cdot t^{m_{\alpha}-1}
(1+t^2)^{-2^{-1}(\lambda_{\alpha}+m_{\alpha})+1}
dsdt\\
=&
\int_0^{\infty}
u^{m_{2\alpha}-1}
(1+u^2)^{-4^{-1}
(\lambda_{\alpha}+m_{\alpha}+2m_{2\alpha})}du\\
&\quad\quad
\cdot 
\int_0^{\infty}
t^{m_{\alpha}-1}
(1+t^2)^{-2^{-1}(\lambda_{\alpha}+m_{\alpha})}
dt.
\end{align*}
By using the integral formula
\[\int_0^{\infty}
x^a(1+x^c)^{-(b+1)}dx
=c^{-1}\Gamma[(a+1)/c]\Gamma[b-((a-c+1)/c)]
/\Gamma(1+b)\]
$({\rm Re}(c)>0;~{\rm Re}(a),{\rm Re}(b)>-1;
~{\rm Re}(b)>{\rm Re}((a-c+a)/c)),$
up to
the constant multiple, this integral is equal to
\[\left(\Gamma(\lambda_{\alpha}/2)
\Gamma((\lambda_{\alpha}+m_{\alpha})/4)\right)
/\left(\Gamma((\lambda_{\alpha}+m_{\alpha})/2)
\Gamma((\lambda_{\alpha}+m_{\alpha}
+2m_{2\alpha})/4)\right).\]
These  calculations
imply the
Gindikin and Karpelevich formula
of the semisimple 
Lie group $\mathrm{F}_{4(-20)}$ which is known
\cite{GK} (cf. \cite[(4.3)]{Os2003}, \cite{Knp2003}).
}\end{remark}

{\bf Acknowledgment.}\quad\quad
The author would like to thank Professor Osami Yasukura 
for his advices and encouragements.


\begin{thebibliography}{nishio second}
\bibitem{D2009}
{\sc  Dijk,~G. van}, 
Introduction to 
Harmonic Analysis and Generalized Gelfand Pairs, 
Walter de Gruyter, Belin New York, 2009.

\bibitem{GK}
 {\sc Gindikin,~S.G. and Karpelevi\v{c},~F.I.},
Planchel measure for Riemannian symmetric spaces of nonpositive
curvature, 
Donklady Akad. Nauk SSSR {\bf 145} (1962),
 252--255;
 English transl., Soviet Math. Dokl. {\bf 3}
 (1962), 962--965. 

\bibitem{Fh1951}
 {\sc Freudenthal,~H.},
Oktaven, Ausnahmergruppen und Oktavengeometrie,
Math. Inst. Rijksuniv. te Utrecht, 1951; 
(a reprint with corrections, 1960) 
= Geometriae Dedicata {\bf 19}-1 (1985), 
7--63.

\bibitem{Fh1953}
 {\sc Freudenthal,~H.},
Zur ebenen Oktavengeomtrie,
Proc. Kon. Ned. Akad. Wer. A. {\bf 56} 
= Indag. Math. {\bf 15} (1953)
 195--200.
 
 \bibitem{HC1958}
 {\sc Harish-Chandra},
Spherical functions on a semi-simple Lie group,
 Amer. J Math.,{\bf 80},(1958)
 241--310.
 
 \bibitem{Hl1970} 
{\sc Helgason,~S.},
A duality for symmetric spaces with application
to group representations,
Advances in Math. {\bf 5} (1970), 1--154.
 
\bibitem{Hl2001} 
{\sc Helgason, ~S.},
{\it Differential Geometry, Lie Groups, and Symmetric Spaces},
American Mathematical Society, 2001.


\bibitem{Knp1986}
{\sc Knapp,~A.W.}, 
Representation Theory of Semisimple Groups,\\
Princeton University Press, 1986.

\bibitem{Knp2003}
{\sc Knapp,~A.W.}, 
The Gindikin-Karpelevi\v{c} Formula and Interwining
Operators,
AMS Transl. (2) Vol. {\bf 210}(2003), 145-159.


\bibitem{Mt1979}
{\sc Matsuki.~T.},
The orbits of affine symmetric spaces under the action
of minimal parabolic subgroups,
J. Math. Soc.  Japan  Vol. 31,  No. 2,~1979.


\bibitem{Mlch1995} 
{\sc Molchanov,~V.F.},
Harmonic Analysis on Homogeneous Spaces,
(Representation Theory 
and Noncommutative Harmonic Analysis II),
Springer-Verlag, 1995.


\bibitem{NY2010}
{\sc Nishio,~A. and Yasukura,~O.},
Orbit Decomposition of Jordan Matrix Algebras of
Order Three under the Automorphism Groups,
J. Math. Sci. Univ. Tokyo
{\bf 17} (2010), 387-417.

\bibitem{Na2012}
{\sc Nishio,~A.}
The classification of orbits
on certain exceptional Jordan algebra
under the automorphism group.
Eprint (2012);
\rm{arXiv:1011.0789v4 [math.DG]}.

\bibitem{OS1980}
{\sc Oshima,~T. and Sekiguchi,~J.},
Eigenspaces of Invariant Differential Operators
on an Affine Symmetric Space,
Inventiones math. 57, 1--81 (1980).

\bibitem{Os2003}
{\sc Oshima,~T.},
A Calculation of {\it c}-functions
for Semisiple Symmetic Spaces,
AMS Transl. (2) Vol. {\bf 210}(2003), 307-330.

\bibitem{Sj2003}
{\sc Sekiguchi,~J.},
The Harish-Chandra's $c$-function of the symmetric spaces,
Symposium Report of Unitary Representation (2003),
26--51. (Japanese)

\bibitem{Sg1971}
{\sc Schiffmann,~G.},
Int\'{e}grales d'entrelacement et fonctions de Whittaker,
Bull soc, math France, {\bf 99} (1971), 3--72.

\bibitem{Tr1979}
{\sc Takahashi,~R.},
Quelques r\'{e}sultats sur l'analyse harmonique
dans l'espace sym\'{e}trique non compact de rang un du
type exceptionnel,
(Analyse harmonique sur les de Lie II),
Lecture Notes in Mathematics {\bf 739},
Springer-Verlag, 1979.

\bibitem{Yi1968}
 {\sc Yokota,~I.},
 A Note on $\mathrm{Spin}(9)$,
 JOUR. FAC. SCI., SHINSHU UNIV. Vol. 3,
 No 1, pp. 61--70, June 1986.

\bibitem{YiJ1973}
 {\sc Yokota,~I.},
Groups and representations, 
Shokabo, Tokyo, 1973 (in Japanese).

\bibitem{Yi1990}
{\sc Yokota,~I.},
Realizations of involutive automorphisms $\sigma$ 
and $G^\sigma$,
of exceptional linear Lie groups $G$, Part I, 
$G=G_2$. $F_4$ and $E_6$,
Tsukuba J. Math. {\bf 14}-1 (1990), 185--223.


\bibitem{Yi_arxiv}
{\sc Yokota,~I.},
Exceptional Lie groups,\\
Eprint (2009);
\rm{arXiv:0902.0431v1 [math.DG]}.

\bibitem{Wja2007}
 {\sc Wolf,~J.A.},
Harmonic Analysis on Commutative Spaces,
 Mathematical Surveys and Monographs Volume 142,
 American Mathematical Society, 2007.
\end{thebibliography}
\end{document}